\documentclass{article}
\usepackage{amsfonts}
\usepackage{amsbsy}
\usepackage{amssymb}
\usepackage{amsmath}
\usepackage{amsthm}
\usepackage{indentfirst}
\usepackage[pdftex]{graphicx}
\usepackage{subfigure}
\usepackage{supertabular}
\usepackage{algorithm}
\usepackage{enumitem}
\usepackage{epstopdf}
\usepackage{listings}
\usepackage{color}
\usepackage{enumerate}
\usepackage{courier}
\usepackage{algorithm}
\usepackage{algpseudocode}
\usepackage{verbatim}
\usepackage[mathcal]{eucal}
\usepackage{multicol}
\usepackage{float}

\theoremstyle{plain}
\newtheorem{definition}{Definition}[section]
\newtheorem{theorem}{Theorem}[section]
\newtheorem{lemma}{Lemma}[section]

\newtheorem{proposition}{Proposition}[section]
\newtheorem{remark}{Remark}[section]
\theoremstyle{definition}

\newcommand{\opn}{\operatorname}
\numberwithin{equation}{section}

\begin{document}

\title{Connections in holomorphic Lie algebroids}
\author{Alexandru IONESCU and Gheorghe MUNTEANU}
\date{}
\maketitle

\begin{abstract}
The main purpose of this note is the study of the total space of a holomorphic
Lie algebroid $E$. The paper is structured in three parts.

In the first section we briefly introduce basic notions on holomorphic Lie
algebroids. The local expressions are written and the complexified
holomorphic bundle is introduced.

The second section is a little broader and includes two approaches to study
the geometry of complex manifold $E.$ The first part contains the study of the tangent bundle $T_{C}E=T^{\prime }E\oplus T^{\prime \prime }E$ and its link, via tangent anchor map, with the complexified tangent bundle $%
T_{C}(T^{\prime }M)=T^{\prime }(T^{\prime }M)\oplus T^{\prime \prime
}(T^{\prime }M).$ A holomorphic Lie algebroid structure has been emphasized
on $T^{\prime }E.$  A special study is made for integral curves of a spray on $%
T^{\prime }E.$ Theorem 2.1 gives the coefficients of a spray, called
canonical, according to a complex Lagrangian on $T^{\prime }E.$
In the second part of section two we study the prolongation $\mathcal{T}%
^{\prime }E$ of $E\times T^{\prime }E$ algebroid structure. 

In the third section we study how a
complex Lagrange (Finsler) structure on $T^{\prime }M$ induces a Lagrangian
structure on $E.$ Three particular cases are analyzed by the rank of anchor
map, the dimensions of manifold $M$ and the fibre dimension. We obtain the
correspondent on $E$ of the well-known (\cite{Mu}) Chern-Lagrange nonlinear connection
from $T^{\prime }M$.

2000 \emph{Mathematics Subject Classification:} 17B66, 53B40.

\emph{Key words:} Holomorphic Lie algebroid, anchor map, spray, nonlinear connection, prolongation,  Lagrangian structures.
\end{abstract}

\section*{Introduction}

Lie algebroids are a generalization of Lie algebras and vector bundles. They
are anchored vector bundles with a Lie bracket defined on the modules of
sections induced from tangent bundle. Lie algebroids provide a natural
setting in which one can develop the theory of differential operators such
as the exterior derivative of forms and the Lie derivative with respect to a
vector field. This setting is slightly more general than that of the tangent
and cotangent bundles of a smooth manifold and their exterior powers.

Lie algebroids represent an active domain of research, with applications in
many areas of mathematics and physics. A well-known example is the work of
A. Weinstein \cite{W} in the area of Mechanics, who developed a generalized
theory of Lagrangians on Lie algebroids and obtained the Euler-Lagrange
equations using the structure of the dual of Lie algebroids and Legendre
transformations associated with a regular Lagrangian. On the other hand, E.
Martinez \cite{Ma1,Ma2} developed another approach using the notion of
prolongation of Lie algebroid, for that fundamental notions on tangent
bundle geometry, such as spray theory and nonlinear connection, can be
transferred to this prolongation. Many recently results are obtained on Lie
algebroids ( \cite{A1,A2,P1,P2, P} etc.).

In complex geometry, some properties of complex and holomorphic Lie
algebroids have been studied in \cite{W1,L-S-X,I-Po}.

The present paper analyzes specific notions from real Lie algebroids theory
in the case of holomorphic Lie algebroids. The paper is organized as follows.
The first part gives basic definitions of a holomorphic anchor map,
holomorphic Lie algebroid, Lie bracket on such an algebroid, which are also locally characterized. More details and linear connections on $E$
and $E_{C}$ are presented in \cite{I}.

In the geometry of the manifold holomorphic Lie algebroid $E$, two
approaches are taken into account. One is of the tangent bundle $T^{\prime }E$,
which has in its turn a natural structure of Lie algebroid. The geometry of $%
T^{\prime }E$ is "linearized" by using a nonlinear connection for which, in
respect to its adapted frames, we study a distinguished complex linear
connection. The classical notions of semisprays and sprays are defined in
this case following the variational problem on an algebroid endowed with a
Lagrangian structure. The main results is Theorem 2.1, which gives the coefficients of a spray, called canonical, from the variational problem.

The second approach concerns the prolongation $\mathcal{T}%
^{\prime}E$ of a holomorphic Lie
algebroid. Using a complete lift we introduce the Liouville tensor and an almost tangent structure for
defining a different type of nonlinear connection on the prolongation $\mathcal{T}%
^{\prime}E$. It is proved how the nonlinear connection on $%
T^{\prime}E$ defines a nonlinear connection on $\mathcal{T}^{\prime}E$. Theorem 2.2 gives the procedure of deriving a nonlinear connection on $\mathcal{T%
}^{\prime}E$ from a spray on $\mathcal{T}^{\prime}E$.  Corroborate with Theorem 2.1, we can say that we have solved the
problem of determining of adapted frames, and so of "linearizing" of the
geometry of a holomorphic algebroid endowed with a regular Lagrangian $L.$

In the last section we study the possibility of inducing Lagrange structures
on holomorphic Lie algebroids from a Lagrangian structure on the tangent
bundle $T^{\prime }M$. Three particular cases are analyzed by the rank of
anchor map and dimensions of manifold $M$ and fiber dimension. It is proved
that a Lagrangian on $T^{\prime }M$ and the well known Chern-Lagrange
nonlinear connection on $T^{\prime }M$ induces a Lagrangian
structure on $\mathcal{T}^{\prime }E$ and consequently, by Theorems 2.1 and
2.2, a nonlinear connection on $\mathcal{T}^{\prime }E$.

\section{Holomorphic Lie algebroids}

\label{S1}

Let $M$ be a complex $n$-dimensional manifold and $E$ a holomorphic vector
bundle of rank $m$ over $M$. Denote by $\pi:E\rightarrow M$ the holomorphic
bundle projection, by $\Gamma(E)$ the module of holomorphic sections of $\pi$
and let $T_\mathbb{C}M = T^{\prime }M\oplus T^{\prime \prime }M$ be the
complexified tangent bundle of $M$, split into the holomorphic and
antiholomorphic tangent bundles.

On a vector bundle $\left( E,\pi ,M\right) $ the definition of a derivative
law is $D:\chi (M)\times \Gamma (E)\rightarrow \Gamma (E)$, $D_{X}s$, such
that $D_{fX}s=fD_{X}s$ and $D_{X}(fs)=fD_{X}s+X(f)$. While these notions
make sense on the fibers of $E$, the Lie bracket $[s_{1},s_{2}]f$, where $%
s_{1},s_{2}\in \Gamma (E)$, has no mathematical meaning. Hence the notion of
Lie algebroids.

\begin{definition}
The holomorphic vector bundle $E$ over $M$ is called anchored if there
exists a holomorphic vector bundle morphism $\rho:E\rightarrow T^{\prime }M$%
, called anchor map.
\end{definition}

Denote by $\Gamma(T^{\prime }M)$ the module of holomorphic sections of $%
T^\prime M$, that is, the holomorphic vector fields on $M$, and by $\mathcal{%
H}(M)$ the ring of holomorphic functions on $M$.

Using the anchor map, we can define a Lie bracket on $E$ from the Lie
bracket on $T^{\prime }M$ by 
\begin{equation}  \label{0}
\rho_E([s_1,s_2]_E) = [\rho_E(s_1),\rho_E(s_2)]_{T^{\prime}M},
\end{equation}
$s_1,s_2\in \Gamma(E)$. For any $f\in \mathcal{H}(M)$, 
\begin{multline*}
\rho_{E}[s_1,fs_2]_{E}=[\rho_{E}(s_1),\rho_{E}(fs_2)]_{T^{\prime}M}=[%
\rho_{E}(s_1),f\rho_{E}(s_2)]_{T^{\prime }M}= \\
=f[\rho_{E}(s_1),\rho_{E}(s_2)]_{T^{\prime }M}+\rho_E (s_1)(f)\rho_E (s_2).
\end{multline*}

These considerations lead to the following definition (\cite%
{W,L-S-X,I-Po,Ma1}):

\begin{definition}
A holomorphic Lie algebroid over $M$ is a triple $(E,[\cdot,\cdot]_E,\rho_E)$%
, where $E$ is a holomorphic vector bundle anchored over $M$, $%
[\cdot,\cdot]_E$ is a Lie bracket on $\Gamma(E)$ and $\rho_E:\Gamma(E)%
\rightarrow\Gamma(T^{\prime }M)$ is the homomorphism of complex modules
induced by the anchor map $\rho$ such that
\begin{equation}  \label{1}
[s_1,fs_2]_E = f[s_1,s_2]_E + \rho_E(s_1)(f)s_2
\end{equation}
for all $s_1,s_2\in\Gamma(E)$ and all $f\in\mathcal{H}(M)$.
\end{definition}

Note that \eqref{0} means that $\rho_E:(\Gamma(E),[\cdot,\cdot]_E)%
\rightarrow(\Gamma(T^{\prime }M),[\cdot,\cdot])$ is a complex Lie algebra
homomorphism.

Also, the Lie bracket $[\cdot ,\cdot ]_{E}$ satisfies the Jacobi identity 
\begin{equation}
\lbrack
s_{1},[s_{2},s_{3}]_{E}]_{E}+[s_{2},[s_{3},s_{1}]_{E}]_{E}+[s_{3},[s_{1},s_{2}]_{E}]_{E}=0.
\label{3}
\end{equation}

The main examples of holomorphic Lie algebroids are, of course, offered by
the holomorphic tangent bundle $T^{\prime }M$, or its cotangent bundle $%
T^{\prime \ast }M$. Some other examples can be derived from those presented
in \cite{I-Po}, in the particular cases of complex manifolds (complex
Poisson manifold, complete lift and prolongation of a Lie algebroid, direct
product structure, etc.).

An example of interest is the projective bundle of a complex Finsler
manifold. If $\ (M,F)$ is a complex Finsler manifold (\cite{Mu}), then $%
F:T^{\prime }M\rightarrow R^{+}$ is a real function of position $z\in M$ and
direction $\eta \in T_{z}^{\prime }M.$ Consider homogeneous coordinates $%
[\eta ]$ that determine $P_{z}M,$ the lines bundle in each $z\in \dot{M}.$
The reunion of all these lines gives the projective bundle $PM\cong
T^{\prime }M/_{C\ast },$ which has a natural structure of holomorphic
Lie algebroid by $\rho :[\eta ]\rightarrow \eta .$ Here, things are more
subtle. $PM$ as a complex manifold is isometric with the indicatrix $IM=\cup
_{z\in M}I_{z}M\,,\ $\ where $I_{z}M=\{\eta \in T_{z}^{\prime }M\ /\
F(z,\eta )=1\}.$ If $g_{i\bar{j}}(z,\eta )$ is the metric tensor of the
complex Finsler space (see below in the paper the corresponding notations),
then $\mathcal{G}=g_{i\bar{j}}d\eta ^{i}\otimes d\bar{\eta}^{j}$ is a metric
structure on $IM$, and $h=g_{i\bar{j}}dz^{i}\otimes d\bar{z}^{j}+(\log \
F^{2})_{i\bar{j}}d\eta ^{i}\otimes d\bar{\eta}^{j}$ is a metric structure on 
$PM$ (see \cite{B-K, Po}). Then a metric structure on $T^{\prime }M$,
descending from $h$ is $g=g_{i\bar{j}}dz^{i}\otimes d\bar{z}^{j}+g_{i\bar{j}%
}\delta \eta ^{i}\otimes \delta \bar{\eta}^{j}.$

\subsection*{Local expressions}

\label{S2}

If $(z^k)_{k=\overline{1,n}}$ is a local complex coordinate system on $%
U\subset M$ and $\{e_\alpha\}_{\alpha=\overline{1,m}}$ is a local frame of
sections of $E$ on $U$, then $(z^k,u^\alpha)$ are local complex coordinates
on $\pi^{-1}(U)\subset E$, where $e = u^\alpha e_\alpha(z)$, $e\in E$.

Let $g_{UV}:U\cap V\rightarrow GL(m,\mathbb{C})$ be the holomorphic
transition functions of $E$. In $z\in U\cap V$, $g_{UV}(z)$ is represented
by the complex matrix of holomorphic functions $\big(M^\alpha_\beta(z)\big)$%
, such that, if $(\widetilde{z}^k,\widetilde{u}^\alpha)$ are local
coordinates on $\pi^{-1}(V)$, then these change by the rules 
\begin{equation}  \label{7.1.1}
\widetilde{z}^k = \widetilde{z}^k(z),\qquad \widetilde{u}^\alpha =
M^\alpha_\beta(z)u^\beta.
\end{equation}

The Jacobi matrix of the transformation laws \eqref{7.1.1} is 
\begin{equation}  \label{7.1.2}
\left( 
\begin{array}{cc}
\dfrac{\partial\widetilde{z}^k}{\partial z^h} & 0 \\ 
&  \\ 
\dfrac{\partial M^\alpha_\beta}{\partial z^h}u^\beta & M^\alpha_\beta%
\end{array}
\right)
\end{equation}

Let $\big(W^\beta_\alpha\big)$ be the inverse matrix of $\big(M^\alpha_\beta%
\big)$, and $\{e_\alpha\}$ a base of sections on $E$, that is, $u=u^\alpha
e_\alpha$ for any $u\in\Gamma(E)$. Then these change by the rules 
\begin{equation*}
\widetilde{e}_\alpha = W^\beta_\alpha e_\beta.
\end{equation*}

The action of the holomorphic anchor map $\rho_E$ can locally be described
by 
\begin{equation}  \label{5}
\rho_E(e_\alpha) = \rho^k_\alpha\frac{\partial}{\partial z^k},
\end{equation}
while the Lie bracket $[\cdot,\cdot]_E$ is locally given by 
\begin{equation}  \label{6}
[e_\alpha,e_\beta]_E = C^{\:\gamma}_{\alpha\beta}e_\gamma.
\end{equation}
The holomorphic functions $\rho^k_\alpha = \rho^k_\alpha(z)$ and $%
C^{\:\gamma}_{\alpha\beta} = C^{\:\gamma}_{\alpha\beta}(z)$ on $M$ are
called \emph{the holomorphic structure functions} of the Lie algebroid E. A
change of local charts on $E$ implies 
\begin{equation}  \label{sch.ro}
\widetilde{\rho}^k_\alpha = W^\beta_\alpha \rho^h_\beta \dfrac{\partial%
\widetilde{z}^k}{\partial z^h}.
\end{equation}

Since $E$ is a holomorphic vector bundle, it has the structure of a complex
manifold, and the natural complex structure acts on its sections by $%
J_{E}(e_{\alpha })=ie_{\alpha }$ and $J_{E}(\bar{e}_{\alpha })=-i\bar{e}%
_{\alpha }$. Hence, the complexified bundle $E_{\mathbb{C}}$ of $E$
decomposes into $E_{\mathbb{C}}=E^{\prime }\oplus E^{\prime \prime }$. The
sections of $E_{\mathbb{C}}$ are given as usual by $\Gamma (E^{\prime
})=\{s-iJ_{E}s\ |\ s\in \Gamma (E)\}$ and $\Gamma (E^{\prime \prime
})=\{s+iJ_{E}s\ |\ s\in \Gamma (E)\}$, respectively. The local basis of
sections of $E^{\prime }$ is $\{e_{\alpha}\}_{\alpha =1,m}$, while for $%
E^{\prime \prime }$, the basis is represented by their conjugates $\{\bar{e}%
_{\alpha }:=e_{\bar{\alpha}}\}_{\alpha =1,m}.$ Since $\rho_{E}:E\rightarrow
T^{\prime }M$ is a homomorphism of complex modules, it extends naturally to
the complexified bundle by $\rho^{\prime }(e_{\alpha })=\rho _{E}(e_{\alpha
})$ and $\rho ^{\prime \prime }(e_{\bar{\alpha }})=\rho _{E}(e_{\bar{\alpha}%
})$. Thus, the anchor map can be decomposed into $\rho_{E}=\rho^{\prime
}\oplus \rho ^{\prime \prime }$ on the complexified bundle, and since $E$ is
holomorphic, the functions $\rho (z)$ are holomorphic, hence $\rho_{\alpha
}^{\bar{k}}=\rho _{\bar{\alpha}}^{k}=0$ and $\rho _{\bar{\alpha}}^{\bar{k}}=%
\overline{\rho _{\alpha }^{k}}$. Thus, the anchored bundles $(E^{\prime
},\rho ^{\prime },T^{\prime }M)$ and $(E^{\prime \prime},\rho ^{\prime
\prime },T^{\prime \prime }M)$ are complex Lie algebroids (\cite{I-Po}). The
Lie brackets are defined as 
\begin{equation*}
[e_{\alpha },e_{\beta }]^{\prime }=[e_{\alpha },e_{\beta }]_{E}=C_{\alpha
\beta }^{\:\gamma }e_{\gamma };\quad [e_{\bar{\alpha}},e_{\bar{\beta}%
}]^{\prime \prime}=\overline{[e_{\alpha },e_{\beta }]}_{E}=C_{\bar{\alpha}%
\bar{\beta}}^{\bar{\gamma}}e_{\bar{\gamma}},
\end{equation*}
where $C_{\bar{\alpha}\bar{\beta}}^{\bar{\gamma}}=\overline{C_{\alpha
\beta}^{\gamma }}.$ On the complexified bundle $E_{\mathbb{C}}$, we have to
consider also the Lie brackets 
\begin{equation*}
[e_{\alpha },e_{\bar{\beta}}]=C_{\alpha \bar{\beta}}^{\gamma }e_{\gamma
}+C_{\alpha \bar{\beta}}^{\bar{\gamma}}e_{\bar{\gamma}};\quad [e_{\bar{\alpha%
}},e_{\beta }]=C_{\bar{\alpha}\beta }^{\gamma }e_{\gamma }+C_{\bar{\alpha}%
\beta }^{\bar{\gamma}}s_{\bar{\gamma}}.
\end{equation*}
It is obvious that $\overline{[e_{\alpha },e_{\bar{\beta}}]}=[e_{\bar{\alpha}%
},e_{\beta }]$, hence $\overline{C_{\alpha \bar{\beta}}^{\bar{\gamma}}}%
=C_{\alpha \bar{\beta}}^{\bar{\gamma}}$ and $\overline{C_{\alpha \bar{\beta}%
}^{\gamma }}=C_{\bar{\alpha}\beta }^{\bar{\gamma}}$.

\begin{proposition}
The structure functions of the complexified Lie algebroid\newline
$(E_{\mathbb{C}},[\cdot,\cdot],\rho_E)$ satisfy the identities: 
\begin{eqnarray*}
\rho^j_\alpha\frac{\partial\rho^i_\beta}{\partial z^j} - \rho^j_\beta\frac{%
\partial\rho^i_\alpha}{\partial z^j} = \rho^i_\gamma
C^{\:\gamma}_{\alpha\beta},\quad \rho^i_\gamma C^{\:\gamma}_{\alpha\bar{\beta%
}} = -\rho^{\bar{j}}_{\bar{\beta}}\frac{\partial\rho^i_\alpha}{\partial\bar{z%
}^j},\quad \rho^{\bar{i}}_{\bar{\gamma}} C^{\:\bar{\gamma}}_{\alpha\bar{\beta%
}} = \rho^j_\alpha\frac{\partial\rho^{\bar{i}}_{\bar{\beta}}}{\partial z^j},
\\
\rho^{\bar{j}}_{\bar{\alpha}}\frac{\partial\rho^{\bar{i}}_\beta}{\partial%
\bar{z}^j} - \rho^{\bar{j}}_{\bar{\beta}}\frac{\partial\rho^{\bar{i}}_{\bar{%
\alpha}}}{\partial\bar{z}^j} = \rho^{\bar{i}}_{\bar{\gamma}} C^{\:\bar{\gamma%
}}_{\bar{\alpha}\bar{\beta}},\quad \rho^{\bar{i}}_{\bar{\gamma}} C^{\:\bar{%
\gamma}}_{\bar{\alpha}\beta} = -\rho^j_\beta\frac{\partial\rho^{\bar{i}}_{%
\bar{\alpha}}}{\partial z^j},\quad \rho^i_\gamma C^{\:\gamma}_{\bar{\alpha}%
\beta} = \rho^{\bar{j}}_{\bar{\alpha}}\frac{\partial\rho^i_\beta}{\partial%
\bar{z}^j}.
\end{eqnarray*}
\end{proposition}

\begin{proof}
The identities follow by direct computations using \eqref{0}, \eqref{5} and \eqref{6}.
\end{proof}

\section{The geometry of the total space of $E$}

Two approaches on the tangent bundle of a holomorphic Lie algebroid $E$ will be described in this section. The first is the classical study of the tangent bundle of $E$, while the second is that of the prolongation on $E$. The latter idea appeared from the need of introducing geometrical objects such as nonlinear connections or sprays which could be studied in a similar manner to the tangent bundle of a complex manifold.

\subsection{The tangent bundle of a holomorphic Lie algebroid}

\label{S3}

Recall (\cite{Mu}) that a complex Lagrange space is a pair $(M,L)$, where $%
L:T^{\prime }M\rightarrow\mathbb{R}$ is a regular Lagrangian defined on the
holomorphic tangent bundle of a complex manifold. The geometrical objects
acting on such a space are sections in the complexified tangent bundle $T_{%
\mathbb{C}}(T^{\prime }M)=T^{\prime }(T^{\prime }M)\oplus T^{\prime \prime
}(T^{\prime }M)$.

The holomorphic tangent bundle $T^{\prime }M$ of $M$ is in its turn a
complex manifold, and the changes of local coordinates $(z^h,\eta^h)$ to $%
(z^{\prime k},\eta^{\prime k})$ are 
\begin{equation}
z^{\prime k}=z^{\prime k}(z)\ ,\quad \eta ^{\prime k}=\frac{\partial
z^{\prime k}}{\partial z^{h}}\eta^{h}.  \label{7.1.1.1}
\end{equation}

The natural frame $\bigg\{\dfrac{\partial }{\partial z^{k}},\dfrac{\partial 
}{\partial \eta ^{k}}\bigg\}$ of $T_{(z,\eta )}^{\prime }(T^{\prime }M)$ in
a fixed point, changes from $(z^{h},\eta ^{h})$ to $(z^{\prime
k},\eta^{\prime k})$ by the rules 
\begin{align}
\dfrac{\partial }{\partial z^{h}}& =\dfrac{\partial z^{\prime k}}{\partial
z^{h}}\dfrac{\partial }{\partial z^{\prime k}}+\dfrac{\partial ^{2}z^{\prime
k}}{\partial z^{j}\partial z^{h}}\eta ^{j}\dfrac{\partial }{\partial
\eta^{\prime k}},  \label{sch.t'm} \\
\dfrac{\partial }{\partial \eta ^{h}}& =\dfrac{\partial z^{\prime k}}{%
\partial z^{h}}\dfrac{\partial }{\partial \eta ^{\prime k}}.  \notag
\end{align}

The generalization consists in introducing a Lagrange structure (in
particular, Finsler) on a holomorphic vector bundle, a well-known idea from T. Aikou 
\cite{A}, and G. Munteanu \cite{Mu}. The basis manifold of such a space is the complex
manifold $E$ endowed with a regular Lagrangian $L:E\rightarrow\mathbb{R}$
and the geometry of the space obviously implies studying geometrical objects
(vectors, metric structures, connections) which act on sections in the
complexified tangent bundle $T_{\mathbb{C}}E=T^{\prime }E\oplus T^{\prime
\prime }E$, where $T^{\prime }E$ is the holomorphic tangent bundle and $%
T^{\prime \prime }E=\overline{T^{\prime }E}$. In particular, a Lie algebroid
is first of all a holomorphic vector bundle and its geometry must be studied.

On $T^{\prime }E$, a natural frame of fields is $\bigg\{\dfrac{\partial }{%
\partial z^{k}},\dfrac{\partial }{\partial u^{\alpha }}\bigg\}$, which, due
to the \eqref{7.1.2} matrix, changes by the rules 
\begin{align}
\dfrac{\partial }{\partial z^{h}}& =\dfrac{\partial \widetilde{z}^{k}}{%
\partial z^{h}}\dfrac{\partial }{\partial \widetilde{z}^{k}}+\dfrac{\partial
M_{\beta }^{\alpha }}{\partial z^{h}}u^{\beta }\dfrac{\partial }{\partial 
\widetilde{u}^{\alpha }},  \label{3.1.5} \\
\dfrac{\partial }{\partial u^{\beta }}& =M_{\beta }^{\alpha }\dfrac{\partial 
}{\partial \widetilde{u}^{\alpha }}.  \notag
\end{align}
Since $E$ is a complex manifold, it follows that $\bigg\{\dfrac{\partial }{%
\partial \overline{z}^{k}},\dfrac{\partial }{\partial \overline{u}^{\alpha }}%
\bigg\}$ is a local frame on $T^{\prime \prime }E=\overline{T^{\prime }E}$
and its rules of change are deduced from \eqref{3.1.5} by conjugation.

Now, let us consider $E$ and $T^{\prime }M$ as manifolds and we prove that
the anchor map $\rho_E$ maps the local coordinates $(z^{k},u^{\alpha })$ on $%
E,$ with changes (\ref{7.1.1}), in a local map $(z^{k},\eta ^{k})$ on $%
T^{\prime }M$, with changes (\ref{7.1.1.1}). Further, let us consider the
same local charts on $M$ for $E$ and $T^{\prime }M$, that is, we have the
same changes $\widetilde{z}^{k}(z)=z^{\prime k}(z)$.

As a mapping between manifolds, the holomorphic anchor $\rho$ induced by $%
\rho_E$ maps $(z^{k},u^{\alpha})$ on $E$ to $(z^{k},\eta^{k})$ on $T^{\prime
}M$, where we define the directional coordinates by 
\begin{equation}
\eta^{k}=u^{\alpha }\rho_{\alpha}^{k}(z).  \label{II.1}
\end{equation}
Let us prove that (\ref{II.1}) define a sistem of coordinates on $%
T^{\prime}M.$ A change of local charts implies that $(\widetilde{z}^{k},%
\widetilde{u}^{\alpha })$ is mapped to $(z^{\prime k},\eta ^{\prime k})$,
where $z^{\prime k}=\widetilde{z}^{k}(z)=z^{\prime k}(z)$ and 
\begin{equation*}
\eta ^{\prime k}=\widetilde{u}^{\alpha }\widetilde{\rho }_{\alpha }^{k}(%
\widetilde{z})=M_{\beta }^{\alpha }u^{\beta }W_{\alpha }^{\gamma }\rho
_{\gamma }^{h}\dfrac{\partial \widetilde{z}^{k}}{\partial z^{h}}=u^{\beta
}\rho _{\gamma }^{h}\dfrac{\partial \widetilde{z}^{k}}{\partial z^{h}}%
=\eta^{h}\dfrac{\partial z^{\prime k}}{\partial z^{h}},
\end{equation*}
so that the changes (\ref{7.1.1.1}) are satisfied, and moreover we have: 
\begin{equation}  \label{sch.coord.}
z^{\prime k} = z^{\prime k}(z),\qquad \eta ^{\prime k} = u^{\gamma }\rho
_{\gamma }^{h}\dfrac{\partial z^{\prime k}}{\partial z^{h}}.
\end{equation}

These transformation laws have the following Jacobi matrix: 
\begin{equation}  \label{J}
\left( 
\begin{array}{cc}
\dfrac{\partial z^{\prime k}}{\partial z^j} & 0 \\ 
&  \\ 
\dfrac{\partial}{\partial z^j}\left(\rho^h_\gamma\dfrac{\partial z^{\prime k}%
}{\partial z^h}\right) u^\gamma & \rho^h_\beta\dfrac{\partial z^{\prime k}}{%
\partial z^h}%
\end{array}
\right)
\end{equation}

Denote by $\rho_*:T_{\mathbb{C}}E\rightarrow T_{\mathbb{C}}(T^{\prime }M)$
the tangent mapping of the anchor $\rho_E:\Gamma(E)\rightarrow\Gamma(T^{%
\prime }M)$ and by $J^*_{T^{\prime}M}:T_{\mathbb{C}}(T^{\prime}M)\rightarrow
T_{\mathbb{C}}(T^{\prime }M)$ the natural complex structure on $T_{\mathbb{C}%
}(T^{\prime }M)$.

\begin{definition}
A complex structure $J_{E}^{\ast }$ on the complex tangent bundle $T_{%
\mathbb{C}}E$ is an endomorphism $J_{E}^{\ast }:T_{\mathbb{C}}E\rightarrow
T_{\mathbb{C}}E$ given by 
\begin{align}
J_{E}^{\ast }\bigg(\dfrac{\partial }{\partial z^{k}}\bigg)=i\dfrac{\partial 
}{\partial z^{k}},\quad & J_{E}^{\ast }\bigg(\dfrac{\partial }{\partial 
\overline{z}^{k}}\bigg)=-i\dfrac{\partial }{\partial \overline{z}^{k}},
\label{act.str.cplx.} \\
J_{E}^{\ast }\bigg(\dfrac{\partial }{\partial u^{\alpha }}\bigg)=i\dfrac{%
\partial }{\partial u^{\alpha }},\quad & J_{E}^{\ast }\bigg(\dfrac{\partial 
}{\partial \overline{u}^{\alpha }}\bigg)=-i\dfrac{\partial }{\partial 
\overline{u}^{\alpha }}.  \notag
\end{align}
\end{definition}

The complex structure $J^*_E$ satisfies the identities ${J^*_E}^2 = -\opn{Id}_{T_{\mathbb{C}}E}$ and $J^*_{T^{\prime }M}\circ\rho_* = \rho_*\circ J^*_E$.

The splitting $T_{\mathbb{C}}E=T^{\prime }E\oplus T^{\prime \prime }E$ of
the complexified tangent bundle is due to the complex structure $J_{E}^{\ast
}$, the holomorphic and antiholomorphic tangent bundles of $E$ corresponding
to the eigenvalues $\pm i$ of $J_{E}^{\ast }$. Moreover, there are two
modules of sections on $T_{\mathbb{C}}E$, $\Gamma(T^{\prime *}_Es\ |\
s\in\Gamma(T_{\mathbb{C}}E)\}$ and $\Gamma(T^{\prime \prime *}_Es\ |\
s\in\Gamma(T_{\mathbb{C}}E)\}$. Since the tangent anchor map $\rho_{\ast }$
is holomorphic, if $s\in \Gamma (T^{\prime }E)$ is a holomorphic section on $%
T_{\mathbb{C}}E$, then $\rho _{\ast }(s)\in \Gamma (T^{\prime}(T^{\prime
}M)))$. Similarly, for an antiholomorphic section $\overline{s}\in \Gamma
(T^{\prime \prime }E)$, $\rho _{\ast }(\overline{s})\in \Gamma(T^{\prime
\prime }(T^{\prime }M))$.

The anchor $\rho$ maps the coordinates $(z^{k},u^{\alpha })$ from a local
chart on the manifold $E$ to the coordinates $(z^{k},\eta^{k}=u^{\alpha
}\rho_{\alpha }^{k}(z))$ in a local chart on $\rho(E)\subset T^{\prime }M$.
The Jacobi matrix of the morphism is $\rho$ is 
\begin{equation*}
\left( 
\begin{array}{cc}
\delta_{h}^{k} & 0 \\ 
&  \\ 
\dfrac{\partial \rho _{\alpha }^{k}}{\partial z^{h}}u^{\alpha } & \rho
_{\alpha }^{h}%
\end{array}
\right)
\end{equation*}

Then $\bigg\{\dfrac{\partial }{\partial z^{k}},\dfrac{\partial }{\partial
\eta ^{k}},\dfrac{\partial }{\partial \overline{z}^{\prime k}},\dfrac{%
\partial }{\partial \overline{\eta }^{k}}\bigg\}$ is the natural frame field
on $T_{\mathbb{C}}(T^{\prime }M)$ and the action of the tangent mapping $%
\rho_{\ast }$ is locally described on $\rho (E)$ by 
\begin{align}
\rho _{\ast }\bigg(\dfrac{\partial }{\partial z^{k}}\bigg)& =:\dfrac{%
\partial ^{\ast }}{\partial z^{k}}=\dfrac{\partial }{\partial z^{k}}
+u^{\alpha }\dfrac{\partial \rho _{\alpha }^{h}}{\partial z^{k}}\dfrac{%
\partial }{\partial \eta ^{h}},  \label{act.apl.tg.} \\
\rho _{\ast }\bigg(\dfrac{\partial }{\partial u^{\alpha }}\bigg)& =:\dfrac{%
\partial ^{\ast }}{\partial u^{\alpha }}=\rho _{\alpha }^{h}\dfrac{\partial 
}{\partial \eta ^{h}}  \notag
\end{align}
and their conjugates. The dual basis of the natural frame $\bigg\{\dfrac{%
\partial }{\partial z^{k}},\dfrac{\partial }{\partial \eta ^{k}}\bigg\}$
induced by $\rho _{\ast }$ on $\rho (E)$ is
\begin{align}
d^{\ast }z^{k} &= dz^{k}  \label{dual} \\
d^{\ast }\eta ^{k} &= u^{\alpha }\dfrac{\partial \rho _{\alpha }^{k}}{%
\partial z^{h}}dz^{h}+\rho_{\alpha }^{k}du^{\alpha }  \notag
\end{align}

For a change of coordinates on $T_{\mathbb{C}}(T^{\prime }M)$, the change
laws on $T_{\mathbb{C}}E$ are, due to the Jacobi matrix \eqref{J}, 
\begin{align}
\dfrac{\partial ^{\ast }}{\partial z^{j}}& =\dfrac{\partial z^{\prime k}}{%
\partial z^{j}}\dfrac{\partial }{\partial z^{\prime k}}+\dfrac{\partial }{%
\partial z^{j}}\left( \rho _{\gamma }^{h}\dfrac{\partial z^{\prime k}}{%
\partial z^{h}}\right) u^{\gamma }\dfrac{\partial }{\partial \eta ^{\prime k}%
},  \label{sch.coord.ro} \\
\dfrac{\partial ^{\ast }}{\partial u^{\beta }}& =\rho _{\beta }^{h}\dfrac{%
\partial z^{\prime k}}{\partial z^{h}}\dfrac{\partial }{\partial \eta
^{\prime k}}  \notag
\end{align}%
and the conjugates.

\subsubsection{The Lie algebroid structure of $T^{\prime }E$}

\label{Ss3.1}

We prove that $T^{\prime }E$ has a Lie algebroid structure over the basis
manifold $M$. Let $p^{\prime }:T^{\prime }M\rightarrow M$ be the projection
of the holomorphic tangent bundle of $M$ and $p_{\ast }^{\prime }:T^{\prime
}T^{\prime }M\rightarrow T^{\prime }M$ its tangent map, acting at a point $%
(z,\eta)$ by $Z^{k}\dfrac{\partial }{\partial z^{k}}+V^{k} \dfrac{\partial }{%
\partial \eta ^{k}}\ \xrightarrow{\ p'_*\ } Z^{k}\dfrac{\partial }{\partial
z^{k}}$.

Then, the following diagram 
\begin{equation}
\begin{array}{ccc}
T^{\prime }E & \xrightarrow{\rho_*} & T^{\prime }T^{\prime }M \\ 
{\scriptstyle\pi _{E}}\downarrow &  & \downarrow {\scriptstyle %
p_{\ast}^{\prime }} \\ 
\ E & \xrightarrow{\ \rho\ } & T^{\prime }M \\ 
{\scriptstyle\pi }\downarrow &  & \downarrow {\scriptstyle p^{\prime }} \\ 
M & \xrightarrow{\operatorname{Id}_M} & M%
\end{array}%
\end{equation}
suggests the definition of the map $\Upsilon :T^{\prime }E\rightarrow
T^{\prime }M$, $\ \Upsilon =p_{\ast }^{\prime }\circ \rho _{\ast }$, in
order to introduce a holomorphic Lie algebroid structure on $T^{\prime }E$.
Since $T^{\prime }M$ and $T^{\prime }E$ are holomorphic bundles, from the
definition of $\Upsilon$ follows that it is a vector bundle morphism.

Locally, we have $Z = Z^{k}\dfrac{\partial }{\partial z^{k}} + V^{\alpha }%
\dfrac{\partial }{\partial u^{\alpha }}\xrightarrow{\ \rho _{\ast }\ }\
Z^{\ast } = Z^{k}\dfrac{\partial ^{\ast }}{\partial z^{k}} + V^{\alpha }%
\dfrac{\partial^{\ast }}{\partial u^{\alpha }}$. The action %
\eqref{act.apl.tg.} and the definition of $p^{\prime }_*$ yield $\Upsilon
(Z) = Z^{k}\dfrac{\partial }{\partial z^{k}}$.

Since $T^{\prime }E$ is a vector bundle over $M$, taking the Lie bracket of
two sections $[Z,W]_{T^{\prime }E}$ and $f:M\rightarrow\mathbb{C}$, we
obtain $\dfrac{\partial f}{\partial u^{\alpha }}=0$, such that 
\begin{equation}  \label{croset}
[Z,fW]_{T^{\prime }E} = f[Z,W] _{T^{\prime }E}+\Upsilon(Z)f\ W.
\end{equation}

This leads to the following

\begin{proposition}
The holomorphic tangent bundle $T^{\prime }E$ has a structure of a Lie
algebroid over the complex manifold $M$, with the anchor map $\Upsilon$.
\end{proposition}

Using the definition of $\Upsilon$, it follows that it is a homomorphism
between the complex Lie algebras $\big(\Gamma(T^{\prime
}E),[\cdot,\cdot]_{T^{\prime }E}\big)$ and $\big(\Gamma(T^{\prime
}M),[\cdot,\cdot]\big)$, that is, 
\begin{equation*}
\Upsilon[Z,W]_{T^{\prime }E} = [\Upsilon(Z),\Upsilon(W)],\quad \forall
Z,W\in\Gamma(T^{\prime }E).
\end{equation*}

\subsubsection{Nonlinear connections on $T^{\prime }E$}

\label{Ss3.2}

It is obvious that the rules of change of the natural frame of fields on $T_{%
\mathbb{C}}E$ are complicated. As in the case of Finsler geometry, the
solution to this problem is the method of nonlinear connection. Consider $%
\pi_*$ the tangent mapping of the projection $\pi:E\rightarrow M$. Then the 
\emph{vertical holomorphic tangent bundle} of $E$ can be defined by $VE
=\ker\pi_*$. A local frame of fields on $VE$ is $\bigg\{\dfrac{\partial}{%
\partial u^\alpha}\bigg\}_{\alpha=\overline{1,m}}$ and if $%
\pi^*(T^{\prime}M) $ is the pull-back bundle of the holomorphic tangent
bundle of $M$, then the following fundamental sequence is obtained (\cite{Mu}%
): 
\begin{equation}  \label{7.1.3}
0\ \rightarrow\ VE\ \overset{i}{\rightarrow}\ T^{\prime }E\ \overset{d\pi}{%
\rightarrow}\ \pi^*(T^{\prime }M)\ \rightarrow\ 0.
\end{equation}

As usual, a splitting $C:T^{\prime }E\rightarrow VE$ in this sequence is
called a \emph{connection} on the vertical bundle and it determines the
decomposition 
\begin{equation}  \label{7.1.4}
T^{\prime }E = VE \oplus HE
\end{equation}
of the holomorphic tangent bundle of $E$, where $HE$ is the \emph{horizontal
distribution}, isomorphic to the pull back bundle $\pi^*(T^{\prime }M)$ by
the morphism $d\pi$ from the exact sequence \eqref{7.1.3}. This isomorphism
is called \emph{complex nonlinear connection} or \emph{Ehresmann connection}
(T. Aikou, \cite{A}) on the holomorphic vector bundle $E$.

The decomposition of the complexified tangent bundle $T_{\mathbb{C}}E$ is
obtained by conjugation:
\begin{equation*}
T_{\mathbb{C}}E = HE \oplus VE \oplus \overline{HE} \oplus \overline{VE}.
\end{equation*}

The \emph{horizontal lift} $l^h:\pi^*(T^{\prime }M)\rightarrow HE$
determined by the nonlinear connection is defined by 
\begin{equation}  \label{7.1.6}
l^h\bigg(\dfrac{\partial}{\partial z^k}\bigg) = \dfrac{\delta}{\delta z^k} = 
\dfrac{\partial}{\partial z^k} - N^\alpha_k\dfrac{\partial}{\partial u^\alpha%
},
\end{equation}
(see \cite{Mu}), where the functions $N^\alpha_k(z,u)$ are called the \emph{%
coefficients of the complex nonlinear connection} on $E$. A change of local
coordinates implies that $\dfrac{\delta}{\delta z^h}$ changes by the rule 
\begin{equation*}
\dfrac{\delta}{\delta z^h} = \dfrac{\partial\widetilde{z}^k}{\partial z^h}%
\dfrac{\delta}{\delta\widetilde{z}^k},
\end{equation*}
such that, using also \eqref{3.1.5}, the laws of change for the functions $%
N^\alpha_k$ are obtained: 
\begin{equation}  \label{sch.coef.}
\dfrac{\partial\widetilde{z}^k}{\partial z^h}\widetilde{N}^\alpha_k =
M^\alpha_\beta N^\beta_h - \dfrac{\partial M^\alpha_\beta}{\partial z^h}%
u^\beta.
\end{equation}

A field of frames $\bigg\{\dfrac{\delta}{\delta z^k},\dfrac{\partial}{%
\partial u^\alpha}\bigg\}$ on $T^{\prime }E$ is obtained, called the \emph{%
adapted frame} of the complex nonlinear connection. A simple computation
using \eqref{croset} and \eqref{7.1.6} leads to the following result.

\begin{proposition}
The Lie brackets of the adapted frame on $T^{\prime }E$ are 
\begin{align}
\bigg[\dfrac{\delta}{\delta z^k},\dfrac{\delta}{\delta z^h}\bigg]_{T^{\prime
}E} &= \bigg(\dfrac{\partial N^\alpha_k}{\partial z^h} - \dfrac{\partial
N^\alpha_h}{\partial z^k}\bigg)\dfrac{\partial}{\partial u^\alpha};
\label{7.1.10} \\
\bigg[\dfrac{\delta}{\delta z^k},\dfrac{\partial}{\partial u^\beta}\bigg]%
_{T^{\prime }E} &= \bigg[\dfrac{\partial}{\partial u^\alpha},\dfrac{\partial%
}{\partial u^\beta}\bigg]_{T^{\prime }E} = 0.  \notag
\end{align}
\end{proposition}

Let $\bigg\{\dfrac{\delta}{\delta z^k},\dfrac{\partial}{\partial\eta^k}%
\bigg\}$ be the adapted frame of a complex nonlinear connection on $%
T^{\prime }T^{\prime }M$, where 
\begin{equation}  \label{3.1.9}
\dfrac{\delta}{\delta z^k} = \dfrac{\partial}{\partial z^k} - N^h_k\dfrac{%
\partial}{\partial\eta^h}
\end{equation}
and denote by $\{dz^k,\delta\eta^k\}$ the dual basis, with 
\begin{equation}  \label{3.1.16}
\delta\eta^k = d\eta^k + N^k_h dz^h.
\end{equation}
The coefficients of the complex nonlinear connection change by the rules (%
\cite{Mu}) 
\begin{equation}  \label{3.1.11}
N^{\prime j}_h\dfrac{\partial z^{\prime h}}{\partial z^k} = \dfrac{\partial
z^{\prime j}}{\partial z^h}N^h_k - \dfrac{\partial^2 z^{\prime j}}{\partial
z^h\partial z^k}\eta^h.
\end{equation}

Analogously, on $T^{\prime }E$, the dual basis of the adapted frame is $%
\{dz^{k},\delta u^{\alpha }\}$, where 
\begin{equation}
\delta u^{\alpha }=du^{\alpha }+N_{h}^{\alpha }dz^{h}.  \label{3.1.16/E}
\end{equation}

From now on, we will use the well-known abbreviations 
\begin{equation*}
\delta_k = \dfrac{\delta}{\delta z^k},\quad \partial_k = \dfrac{\partial}{%
\partial z^k},\quad \partial_\alpha = \dfrac{\partial}{\partial u^\alpha}.
\end{equation*}

\subsubsection{Linear connections on $T^{\prime }E$}

\label{Ss3.3}

Linear connections can be introduced on the holomorphic tangent bundle $%
T^{\prime }E$ of the holomorphic Lie algebroid $E$ in a similar manner as in
the case of the holomorphic tangent bundle $T^{\prime }T^{\prime }M$. Denote
by $A^r$ the set of complex $r$-forms over $M$ and $A^r(E)$, the set of
complex $r$-forms on $E$.

\begin{definition}
A complex linear connection on $T^{\prime }E$ is a map 
\begin{equation*}
D:\Gamma(T^{\prime }E)\times\Gamma(T^{\prime }E)\rightarrow\Gamma(T^{\prime
}E),\qquad (Z,W)\mapsto D_ZW,
\end{equation*}
such that 
\begin{equation}  \label{c.l.c.}
D_Z(fW) = (\Upsilon(Z)f)W + fD_ZW,\quad \forall f\in A^0(E),\ \forall
Z,W\in\Gamma(T^{\prime }E).
\end{equation}
\end{definition}

Locally, a section $Z\in\Gamma(T^{\prime }E)$ can be decomposed in the
adapted frame $\bigg\{\dfrac{\delta}{\delta z^k},\dfrac{\partial}{\partial
u^\alpha}\bigg\}$ of a complex nonlinear connection as introduced in the
previous section. Hence, the connection forms have vertical and horizontal
components. If the linear connection preserves the distributions from %
\eqref{7.1.4}, then it is called \emph{distinguished}. A distinguished
complex linear connection $D$ on $T^{\prime }E$ has the following
coefficients: 
\begin{equation}  \label{coef.d-clc}
D_{\delta_k}\delta_j = L^{\:i}_{jk}\delta_i,\quad
D_{\partial_\gamma}\delta_j = L^{\:i}_{j\gamma}\delta_i,\quad
D_{\delta_k}\partial_\beta = L^{\:\alpha}_{\beta k}\partial_\alpha,\quad
D_{\partial_\gamma}\partial_\beta =
C^{\:\alpha}_{\beta\gamma}\partial_\alpha.
\end{equation}

As usual, the next step is considering the torsion and curvature of such a
connection.

As usual, the torsion of a distinguished complex linear connection on $%
T^{\prime }E$ is 
\begin{equation}  \label{curb'}
T(Z,W) = D_ZW - D_WZ - [Z,W].
\end{equation}
Its coefficients are denoted by $T(\delta_h,\delta_k) = T^{\:
i}_{hk}\delta_i + T^{\:\alpha}_{hk}\delta_\alpha$, etc. They are given by 
\begin{align*}
T^{\:i}_{hk} &= L^{i}_{kh} - L^{i}_{hk}, \\
T^{\:\alpha}_{hk} &= \partial_kN^\alpha_h - \partial_hN^\alpha_k, \\
T^{\:i}_{h\alpha} &= -L^{\:i}_{h\alpha}, \\
T^{\:\beta}_{h\alpha} &= L^{\:\beta}_{\alpha h}, \\
T^{\:\gamma}_{\alpha\beta} &= C^{\:\gamma}_{\alpha\beta} -
T^{\:\gamma}_{\beta\alpha}.
\end{align*}

The curvature of a distinguished complex linear connection on $T^{\prime }E$
is defined by 
\begin{equation}  \label{tors}
R(Z,W) = D_ZD_W - D_WD_Z - D_{[Z,W]}.
\end{equation}
In the adapted frame of fields, the coefficients of the curvature are 
\begin{align*}
R^{\ i}_{jhk} &= \partial_kL^{\:i}_{jh} - \partial_hL^{\:i}_{jk} +
L^{\:l}_{jh}L^{\:i}_{lk} - L^{\:l}_{jk}L^{\:i}_{lh} -
(\partial_hN^\alpha_k-\partial_kN^\alpha_h)L^{\:i}_{j\alpha}, \\
R^{\ \alpha}_{\beta hk} &= \partial_hL^{\:\alpha}_{\beta k} -
\partial_kL^{\:\alpha}_{\beta h} + L^{\:\gamma}_{\beta
k}L^{\:\alpha}_{\gamma h} + L^{\:\gamma}_{\beta h}L^{\:\alpha}_{\gamma k} -%
\big(\partial_hN^\gamma_k - \partial_kN^\gamma_h\big)C^{\:\alpha}_{\gamma%
\beta}, \\
R^{\ \alpha}_{\gamma k\beta} &= \partial_kC^{\:\alpha}_{\gamma\beta} +
C^{\;\sigma}_{\gamma\beta}L^{\:\alpha}_{\sigma k} - L^{\:\sigma}_{\gamma
k}C^{\:\alpha}_{\sigma\beta}, \\
R^{\ i}_{hk\beta} &= \partial_kL^{\: i}_{h\beta} + L^{\: j}_{h\beta}L^{\:
i}_{jk} - L^{\: j}_{hk}L^{\: i}_{j\beta}, \\
R^{\ i}_{k\alpha\beta} &= L^{\: j}_{k\beta}L^{\: i}_{j\beta} - L^{\:
j}_{k\alpha}L^{\: i}_{j\beta}, \\
R^{\ \sigma}_{\gamma\alpha\beta} &=
C^{\:\tau}_{\gamma\beta}C^{\:\sigma}_{\tau\alpha} -
C^{\:\tau}_{\gamma\alpha}C^{\:\sigma}_{\tau\beta}.
\end{align*}

\subsubsection{Semisprays and sprays}

The notion of semispray on a holomorphic Lie algebroid has been introduced
in \cite{I} following the steps from the real case (\cite{A1,A2}). Let $%
\rho_E$ denote the anchor map and $\pi_*$, the tangent map of the projection 
$\pi$ and $\tau_E:T^{\prime }E\rightarrow E$.

\begin{definition}
A holomorphic section $S:E\rightarrow T^{\prime }E$ is called semispray if
\begin{description}
\item[i)] $\tau_E\circ S = \opn{Id}_E$,
\item[ii)] $\pi_*\circ S = \rho_E$.
\end{description}
\end{definition}

Let $c:I\rightarrow M,\ I\subset\mathbb{R}$ be a complex curve on $M$, $%
\widetilde{c}:I\rightarrow E$ a complex curve on $E$ such that $\pi\circ%
\widetilde{c} = c$ and denote by $\dot{\widetilde{c}}$ the tangent vector
field to the curve $\widetilde{c}$.

\begin{definition}
The vector field $\dot{\widetilde{c}}$ is called admissible if 
\begin{equation}
\pi_*(\dot{\widetilde{c}}) = \rho(\widetilde{c}).
\end{equation}
\end{definition}

Locally, $c(t) = (z^k(t))$, $\widetilde{c} = (z^k(t),u^\alpha(t))$ and $\dot{%
\widetilde{c}} = \dfrac{dz^k}{dt}\dfrac{\partial}{\partial z^k} + \dfrac{%
du^\alpha}{dt}\dfrac{\partial}{\partial u^\alpha}$, $t\in I$. Then, the
curve $\dot{\widetilde{c}}$ is admissible if and only if 
\begin{equation*}
\dfrac{dz^k}{dt}(t) = \rho^k_\alpha(z(t))u^\alpha(t),\ \forall t\in I.
\end{equation*}

If $S = Z^k\dfrac{\partial}{\partial z^k} + U^\alpha\dfrac{\partial}{%
\partial u^\alpha}$, then, using the definition, it follows that $S$ is a
semispray if and only if 
\begin{equation}  \label{3.5}
Z^k(z,u) = \rho^k_\alpha(z)u^\alpha.
\end{equation}
The coefficients $U^\alpha(z,u)$ are not determined, thus, for easier
computations, let $U^\alpha = -2G^\alpha$, such that 
\begin{equation}  \label{spray}
S = \rho^k_\alpha u^\alpha\dfrac{\partial}{\partial z^k} - 2G^\alpha(z,u)%
\dfrac{\partial}{\partial u^\alpha}.
\end{equation}

The rules of change for the coordinates of $S$ are obtained using the %
\eqref{7.1.2} matrix: 
\begin{equation}  \label{3.6}
\widetilde{Z}^k = \dfrac{\partial\widetilde{z}^k}{\partial z^h}Z^h
\end{equation}
and
\begin{equation}  \label{3.7}
\widetilde{G}^\alpha = M^\alpha_\beta G^\beta - \dfrac{1}{2}\dfrac{\partial
M^\alpha_\beta}{\partial z^k}u^\beta\rho^k_\gamma u^\gamma.
\end{equation}
Moreover, due to \eqref{sch.ro}, the coefficients $Z^k(z,u)$ given by %
\eqref{3.5} verify the \eqref{3.6} laws of change, which leads to the
following result.

\begin{proposition}
A vector field $S = \rho^k_\alpha u^\alpha\dfrac{\partial}{\partial z^k} -
2G^\alpha\dfrac{\partial}{\partial u^\alpha}\in\Gamma(T^{\prime }E)$ is a
semispray if and only if the coefficients $G^\alpha$ verify the \eqref{3.7}
rules of transformation.
\end{proposition}

A curve $c:t\mapsto (z^i(t),u^\alpha(t))$ on $E$ is an integral curve of the
semispray $S$ if it satisfies the system of differential equations 
\begin{equation}  \label{3.8}
\dfrac{dz^i}{dt} = \rho^k_\alpha(t)u^\alpha,\qquad \dfrac{du^\alpha}{dt} +
2G^\alpha(z,u) = 0.
\end{equation}
A semispray can then be characterized also by

\begin{proposition}
A vector field on $E$ is a semispray if and only if all its integral curves
are admissible.
\end{proposition}

Now, if $h_\lambda:E\rightarrow E$ is the complex homothety $%
h_\lambda:e\mapsto \lambda e,\ \lambda\in\mathbb{C},\ e\in E$, then a
semispray $S$ on $E$ is called spray if 
\begin{equation}  \label{H}
S\circ h_\lambda = \lambda h_{\lambda,*}\circ S.
\end{equation}
Since the action of $h_\lambda$ is locally described by $h_\lambda:(z^k,u^%
\alpha)\mapsto(z^k,\lambda u^\alpha)$, the condition \eqref{H} becomes,
equivalently, 
\begin{equation}  \label{H0}
G^\alpha(z,\lambda u) = \lambda^2G^\alpha(z,u),
\end{equation}
that is, the functions $G^\alpha$ are complex homogeneous of degree $2$ in $%
u $.

Let $L = u^\alpha\dfrac{\partial}{\partial u^\alpha}$ be the complex
Liouville vector field on $E$. Then, an even simpler formulation for the
condition of spray can be obtained using Euler's theorem for homogeneous
functions: 
\begin{equation}
[L,S]_E = S.
\end{equation}

In \cite{I}, we have obtained a complex spray from the variational problem.
Following the ideas of Weinstein (\cite{W}), the Euler-Lagrange equations on 
$E$ are 
\begin{equation}  \label{v2}
\dfrac{d}{dt}\left( \frac{\partial L}{\partial u^{\beta }}\right) =
\rho_{\beta }^k\frac{\partial L}{\partial z^k} + \rho _{\beta }^{\bar{k}}%
\dfrac{\partial L}{\partial \bar{z}^k} + Q_{\beta }^{\alpha }\dfrac{\partial
L}{\partial u^{\alpha }} + Q_{\beta }^{\bar{\alpha}}\dfrac{\partial L}{%
\partial \bar{u}^{\alpha }},
\end{equation}
where $\rho_{\beta }^{\bar{k}} = 0$ since $E$ is holomorphic and $Q_{\beta
}^{\alpha }$ and $Q_{\beta }^{\bar{\alpha}}$ must be determined.

\begin{theorem}
On a holomorphic Lie algebroid $E$ endowed with a regular Lagrangian $L(z,u)$
and a Hermitian metric tensor $g_{\bar{\alpha}\beta}$ with $\det(g_{\bar{%
\alpha}\beta})\neq 0$, a complex canonical spray is given by 
\begin{equation}  \label{s.can.}
G^{\alpha } = \dfrac{1}{2}\left( g^{\bar{\beta}\alpha }\dfrac{\partial ^{2}L%
}{\partial z^k\partial \bar{u}^{\beta }} + \dfrac{1}{2}W_{\varepsilon
}^{\alpha}\dfrac{\partial M_{\beta }^{\varepsilon }}{\partial z^k}u^{\beta
}\right)\rho _{\gamma }^k u^{\gamma }
\end{equation}
\end{theorem}

\begin{remark}
If the Lagrangian on $E$ is complex homogeneous, then the spray is complex
homogeneous of degree $2$ in $u$.
\end{remark}

\subsection{The prolongation of a holomorphic Lie algebroid}

\label{S5}

For the holomorphic Lie algebroid $E$ over a complex manifold $M$, its
prolongation will be introduced using the tangent mapping $\pi_*:T^{\prime
}E\rightarrow T^{\prime }M$ and the holomorphic anchor map $%
\rho_E:E\rightarrow T^{\prime }M$. Define the subset $\mathcal{T}^{\prime }E$
of $E\times T^{\prime }E$ by $\mathcal{T}^{\prime }E = \{(e,v)\in E\times
T^{\prime }E | \rho(e)=\pi_*(v)\}$ and the mapping $\pi_{\mathcal{T}}:%
\mathcal{T}^{\prime }E\rightarrow E$, given by $\pi_{\mathcal{T}}(e,v) =
\pi_E(v)$, where $\pi_E:T^{\prime }E\rightarrow E$ is the tangent
projection. Then $(\mathcal{T}^{\prime }E,\pi_{\mathcal{T}},E)$ is a
holomorphic vector bundle over $E$, of rank $2m$. Moreover, it is easy to
verify that the projection onto the second factor $\rho_{\mathcal{T}}:%
\mathcal{T}^{\prime }E\rightarrow T^{\prime }E$, $\rho_{\mathcal{T}}(e,v) =
v $, is the anchor of a new holomorphic Lie algebroid over the complex
manifold $E$ (see \cite{Ma1,Ma2,P1} for details in the real case).

The vertical subbundle of the prolongation is defined using the projection
onto the first factor $\tau_1:\mathcal{T}^{\prime }E\rightarrow E$, $%
\tau_1(e,v) = e$, by
\begin{equation*}
V\mathcal{T}^{\prime }E = \ker\tau_1 = \{(e,v)\in\mathcal{T}^{\prime }E\ |\
\tau_1(e,v) = 0\}.
\end{equation*}
From the construction above, it follows that any element of $V\mathcal{T}%
^{\prime }E$ has the form $(0,v)\in E\times\mathcal{T}^{\prime }E$, with $%
\pi_*(v) = 0$. Then, vertical elements $(0,v)\in V\mathcal{T}^{\prime }E$
have the property $v\in\ker\pi_*$ and $v$ is a vertical vector on $E$.

The local coordinates on $\mathcal{T}^{\prime }E$ are $(z^k,u^\alpha,v^%
\alpha,w^\alpha)$, obtained from the local coordinates $(z^k,u^\alpha)$ of $%
e $ by using the identity $\rho(e) = \pi_*(v)$, which yields the vector $v$
in the form 
\begin{equation*}
v = \rho^k_\alpha v^\alpha\dfrac{\partial}{\partial z^k} + w^\alpha\dfrac{%
\partial}{\partial u^\alpha}.
\end{equation*}

The local basis of holomorphic sections in $\Gamma(\mathcal{T}^{\prime }E)$
is $\{\mathcal{Z}_\alpha,\mathcal{V}_\alpha,\}$, defined by 
\begin{equation*}
\mathcal{Z}_\alpha(e) = \bigg( s_\alpha(\pi(e)), \rho^k_\alpha\dfrac{\partial%
}{\partial z^k}\bigg|_e \bigg),\qquad \mathcal{V}_\alpha(e) = \bigg( 0, 
\dfrac{\partial}{\partial u^\alpha}\bigg|_e \bigg),
\end{equation*}
where $\bigg\{\dfrac{\partial}{\partial z^k}, \dfrac{\partial}{\partial
u^\alpha}\bigg\}$ is the natural frame on $T^{\prime }E$.

If $W$ is a holomorphic section of $\mathcal{T}^{\prime }E$, then its
decomposition in the basis $\{\mathcal{Z}_\alpha,\mathcal{V}_\alpha\}$ is 
\begin{equation*}
W = Z^\alpha\mathcal{Z}_\alpha + V^\alpha\mathcal{V}_\alpha,
\end{equation*}
where $Z^\alpha$ and $V^\alpha$ are holomorphic functions of $z$ and $u$.

Also, the holomorphic vector field $\rho_{\mathcal{T}}(W)\in\Gamma(T^{\prime
}E)$ can be written as 
\begin{equation*}
\rho_{\mathcal{T}}(W) = \rho^k_\alpha Z^\alpha(z,u)\dfrac{\partial}{\partial
z^k}\bigg|_{(z,u)} + V^\alpha(z,u)\dfrac{\partial}{\partial u^\alpha}\bigg|%
_{(z,u)}.
\end{equation*}

A section $Z\in\Gamma(E)$ can be lifted to sections of the prolongation $%
\mathcal{T}^{\prime }E$ by considering its vertical and complete lifts $Z^v$
and $Z^c$, which will be defined in the following (see \cite{I}). The
vertical lift of a section $Z\in\Gamma(E)$, $Z=Z^\alpha s_\alpha$, is a
vector field on $E$ given by 
\begin{equation}  \label{7}
Z^v(z,u) = Z^\alpha(z)\dfrac{\partial}{\partial u^\alpha}.
\end{equation}
In particular, $s^v_\alpha = \dfrac{\partial}{\partial u^\alpha}$.

The complete lift $Z^c$ of a section $Z\in\Gamma(E)$ is a vector field on $E$
defined by 
\begin{equation}  \label{9}
Z^c(z,u) = Z^\alpha\rho^k_\alpha\dfrac{\partial}{\partial z^k} + \bigg( %
\rho^k_\beta\dfrac{\partial Z^\alpha}{\partial z^k} - Z^\gamma
C^{\:\alpha}_{\gamma\beta} \bigg) u^\beta \dfrac{\partial}{\partial u^\alpha}%
.
\end{equation}

In particular, $s^c_\alpha = \rho^k_\alpha\dfrac{\partial}{\partial z^k} -
C^{\:\gamma}_{\alpha\beta} u^\beta \dfrac{\partial}{\partial u^\gamma}$.

The lifts on $\mathcal{T}^{\prime }E$ are defined as 
\begin{equation*}
Z^V(e) = (0,Z^v(e)),\qquad Z^C(e) = (Z(\pi(e)),Z^c(e)),\qquad e\in E.
\end{equation*}
In local coordinates, if $Z = Z^\alpha s_\alpha$, then the expressions of $%
Z^V$ and $Z^C$ are 
\begin{equation*}
Z^V = Z^\alpha\mathcal{V}_\alpha,\qquad Z^C = Z^\alpha\mathcal{Z}_\alpha + %
\bigg( \rho^k_\beta\dfrac{\partial Z^\alpha}{\partial z^k} - Z^\gamma
C^{\:\alpha}_{\gamma\beta} \bigg) u^\beta\mathcal{V}_\alpha.
\end{equation*}
In particular, $s^V_\alpha = \mathcal{V}_\alpha$ and $s^C_\alpha = \mathcal{Z%
}_\alpha - C^{\:\beta}_{\alpha\gamma}u^\gamma\mathcal{V}_\beta$.

The Lie bracket $[\cdot,\cdot]_{\mathcal{T}}$ on $\mathcal{T}^{\prime }E$
satisfies the identities 
\begin{equation*}
[Z^V,W^V]_{\mathcal{T}} = 0,\qquad [Z^V,W^C]_{\mathcal{T}} =
[Z,W]^V_E,\qquad [Z^C,W^C]_{\mathcal{T}} = [Z,W]^C_E
\end{equation*}
for $Z,W\in\Gamma(E)$. The structure of a holomorphic Lie algebroid on the
vector bundle $(\mathcal{T}^{\prime }E,\pi_{\mathcal{T}},E)$ is therefore
given by $([\cdot,\cdot]_{\mathcal{T}},\rho_{\mathcal{T}})$. The action of
the anchor $\rho_{\mathcal{T}}$ on $\mathcal{T}^{\prime }E$ is locally
described by 
\begin{equation*}
\rho_{\mathcal{T}}(\mathcal{Z}_\alpha) = \rho^k_\alpha\dfrac{\partial}{%
\partial z^k},\qquad \rho_{\mathcal{T}}(\mathcal{V}_\alpha) = \dfrac{\partial%
}{\partial u^\alpha}.
\end{equation*}

\begin{lemma}
The Lie brackets of the basis $\{\mathcal{Z}_\alpha,\mathcal{V}_\alpha\}$
are: 
\begin{equation*}
[\mathcal{Z}_\alpha,\mathcal{Z}_\beta]_\mathcal{T} =
C^{\:\gamma}_{\alpha\beta}\mathcal{Z}_\gamma,\qquad [\mathcal{Z}_\alpha,%
\mathcal{V}_\beta]_\mathcal{T} = 0,\qquad [\mathcal{V}_\alpha,\mathcal{V}%
_\beta]_\mathcal{T} = 0.
\end{equation*}
\end{lemma}

As in the real case (\cite{Ma1}), a differential $\partial_{\mathcal{T}}$
can be defined on $\mathcal{T}^{\prime }E$. Denoting by $\{\mathcal{Z}%
^\alpha,\mathcal{V}^\alpha\}$ the dual base of $\{\mathcal{Z}_\alpha,%
\mathcal{V}_\alpha\}$, then 
\begin{equation*}
\partial_{\mathcal{T}}z^k = \rho^k_\alpha\mathcal{Z}^\alpha,\qquad \partial_{%
\mathcal{T}}u^\alpha = \mathcal{V}^\alpha
\end{equation*}
and 
\begin{equation*}
\partial_{\mathcal{T}}\mathcal{Z}^\alpha = -\dfrac{1}{2}C^{\:\alpha}_{\beta%
\gamma}\mathcal{Z}^\beta\wedge\mathcal{Z}^\gamma,\qquad \partial_{\mathcal{T}%
}\mathcal{V}^\alpha = 0.
\end{equation*}

As announced in the previous section, the notion of semispray for a Lie
algebroid can be introduced in another manner as well (\cite{Ma1,P1}). More
precisely, a semispray can also be considered on the prolongation $\mathcal{T%
}^{\prime }E$ of the holomorphic Lie algebroid $E$. Let $\mathcal{L}$ be the
complex Liouville section on $\mathcal{T}^{\prime }E$, defined by 
\begin{equation}
\mathcal{L}(e) = (0,e^V_e),\quad e\in E.
\end{equation}
The coordinate expression of $\mathcal{L}$ is 
\begin{equation}  \label{s1}
\mathcal{L} = u^\alpha\mathcal{V}_\alpha.
\end{equation}

Also, let $T$ be the tangent structure (or vertical endomorphism)
defined on $\mathcal{T}^{\prime }E$ by 
\begin{equation}  \label{s2}
T(Z^C) = Z^V,\qquad T(Z^V) = 0.
\end{equation}
In local coordinates, 
\begin{equation}  \label{str.tg.}
T = \mathcal{Z}^\alpha\otimes\mathcal{V}_\alpha,
\end{equation}
which yields 
\begin{equation}  \label{act.str.tg.}
T(\mathcal{Z}_\alpha) = \mathcal{V}_\alpha,\qquad T(\mathcal{V}_\alpha) = 0.
\end{equation}

\begin{lemma}
If $T$ is the complex tangent structure on $\mathcal{T}^{\prime }E$ and $C$
is the complex Liouville section, then 
\begin{equation}
T^2 = 0,\qquad \opn{Im}T = \ker T = V\mathcal{T}^{\prime }E,\qquad [%
\mathcal{L},T]_\mathcal{T} = -T.
\end{equation}
\end{lemma}

These two canonical complex objects on $\mathcal{T}^{\prime }E$ can now be
used for defining the notion of complex semispray on the prolongation of $E$.

\begin{definition}
A section $\mathcal{S}$ of the holomorphic Lie algebroid $\mathcal{T}%
^{\prime }E$ is called complex semispray on $E$ if 
\begin{equation*}
T(\mathcal{S}) = \mathcal{L}.
\end{equation*}
\end{definition}

In order to describe locally a complex semispray on $\mathcal{T}^{\prime }E$%
, let $\mathcal{S} = A^\alpha\mathcal{Z}_\alpha + B^\alpha\mathcal{V}_\alpha$%
. Then, \eqref{s1} and \eqref{act.str.tg.} yield $A^\alpha = u^\alpha$, and
for convenience let $B^\alpha = -2G^\alpha$. Therefore, the local expression
of a semispray on $\mathcal{T}^{\prime }E$ is 
\begin{equation*}
\mathcal{S} = u^\alpha\mathcal{Z}_\alpha - 2G^\alpha(z,u)\mathcal{V}_\alpha.
\end{equation*}

If $[\mathcal{L},\mathcal{S}]_\mathcal{T} = \mathcal{S}$, then $\mathcal{S}$
is called spray and $G^\alpha$ are homogeneous functions of degree $2$.

\subsubsection{Nonlinear connections on $\mathcal{T}^{\prime }E$}

\label{S6}

The method of nonlinear connection discussed in Section \ref{Ss3.2} will be
applied here for the prolongation $\mathcal{T}^{\prime }E$ of the
holomorphic Lie algebroid $E$. A complex nonlinear connection on $\mathcal{T}%
^{\prime }E$ is given by a complex vector subbundle $H\mathcal{T}^{\prime }E$
of $\mathcal{T}^{\prime }E$ such that $\mathcal{T}^{\prime }E = H\mathcal{T}%
^{\prime }E \oplus V\mathcal{T}^{\prime }E$. If $l^h$ is the horizontal
lift, then similar considerations as in the real case (\cite{P1}) lead to
the following local expression of
\begin{equation*}
l^h(\mathcal{Z}_\alpha) = \mathcal{Z}_\alpha - N^\beta_\alpha\mathcal{V}%
_\beta,\qquad l^h(\mathcal{V}_\alpha) = 0,
\end{equation*}
where $N^\beta_\alpha = N^\beta_\alpha(z,u)$ are functions defined on $E$,
called the coefficients of the complex nonlinear connection on $\mathcal{T}%
^{\prime }E$.

Denote by
\begin{equation}  \label{7.1}
\delta_\alpha = \mathcal{Z}_\alpha - N^\beta_\alpha\mathcal{V}_\beta
\end{equation}
in order to obtain a local frame $\{\delta_\alpha,\mathcal{V}_\alpha\}$ on $%
\mathcal{T}^{\prime }E$, called the \emph{adapted frame} with respect to the
complex nonlinear connection $N$ on $\mathcal{T}^{\prime }E$. Then 
\begin{equation}  \label{7.2}
\rho_{\mathcal{T}}(\delta_\alpha) = \rho^k_\alpha\dfrac{\partial}{\partial
z^k} - N^\beta_\alpha\dfrac{\partial}{\partial u^\beta},\qquad \rho_{%
\mathcal{T}}(\mathcal{V}_\alpha) = \dfrac{\partial}{\partial u^\alpha}.
\end{equation}

The dual of the adapted frame of fields is $\{\mathcal{Z}^\alpha,\delta%
\mathcal{V}^\alpha\}$, where 
\begin{equation*}
\delta\mathcal{V}^\alpha = \mathcal{V}^\alpha + N^\alpha_\beta\mathcal{Z}%
^\beta
\end{equation*}
and $\{\mathcal{Z}^\alpha,\mathcal{V}^\alpha\}$ is the dual basis of $\{%
\mathcal{Z}_\alpha,\mathcal{V}_\alpha\}$.

\begin{proposition}
The Lie brackets of the adapted frame $\{\delta_\alpha,\mathcal{V}_\alpha\}$
are 
\begin{align*}
[\delta_\alpha,\delta_\beta]_{\mathcal{T}} &=
C^{\:\gamma}_{\alpha\beta}\delta_\gamma + \mathcal{R}^{\:\gamma}_{\alpha%
\beta}\mathcal{V}_\gamma, \\
[\delta_\alpha,\mathcal{V}_\beta]_{\mathcal{T}} &= \dfrac{\partial
N^\gamma_\alpha}{\partial u^\beta}\mathcal{V}_\beta, \\
[\mathcal{V}_\alpha,\mathcal{V}_\beta]_{\mathcal{T}} &= 0,
\end{align*}
where 
\begin{equation*}
\mathcal{R}^{\:\gamma}_{\alpha\beta} =
C^{\:\varepsilon}_{\alpha\beta}N^\gamma_\varepsilon + \rho^k_\beta\dfrac{%
\partial N^\gamma_\alpha}{\partial z^k} - \rho^k_\alpha\dfrac{\partial
N^\gamma_\beta}{\partial z^k} - N^\varepsilon_\beta\dfrac{\partial
N^\gamma_\alpha}{\partial u^\varepsilon} + N^\varepsilon_\alpha\dfrac{%
\partial N^\gamma_\beta}{\partial u^\varepsilon}.
\end{equation*}
\end{proposition}

Now, consider the complex nonlinear connection on $T^{\prime }E$ introduced
in Section \ref{Ss3.2}. Its coefficients, $N^\alpha_k(z,u)$, change by the
rules \eqref{sch.coef.} and the adapted frame of fields is $\bigg\{\dfrac{%
\delta}{\delta z^k},\dfrac{\partial}{\partial u^\alpha}\bigg\}$, where 
\begin{equation*}
\dfrac{\delta}{\delta z^k} = \dfrac{\partial}{\partial z^k} - N^\alpha_k%
\dfrac{\partial}{\partial u^\alpha}.
\end{equation*}It is interesting to study the relation between the two nonlinear connections on $%
T^{\prime }E$ and $\mathcal{T}^{\prime }E$, respectively. The first relation
in \eqref{7.2} suggests considering another adapted frame on $T^{\prime }E$,
defined by 
\begin{equation}  \label{7.3}
\delta_\alpha = \rho^k_\alpha\dfrac{\partial}{\partial z^k} - N^\beta_\alpha%
\dfrac{\partial}{\partial u^\beta},
\end{equation}
and imposing that it changes by the rules 
\begin{equation}  \label{7.4}
\delta_\alpha = M^\beta_\alpha\widetilde{\delta}_\beta
\end{equation}
or, using \eqref{sch.coef.} and \eqref{3.1.5}, 
\begin{equation}  \label{7.5}
M^\beta_\alpha\widetilde{N}^\gamma_\beta = M^\gamma_\beta N^\beta_\alpha -
\rho^k_\alpha\dfrac{\partial M^\gamma_\beta}{\partial z^k}u^\beta.
\end{equation}
Note that these rules of change can be obtained from \eqref{sch.coef.} by
contrancting with $\rho^k_\alpha$ and by denoting 
\begin{equation}  \label{7.6}
N^\beta_\alpha = \rho^k_\alpha N^\beta_k.
\end{equation}
It follows that

\begin{proposition}
A complex nonlinear connection on $T^{\prime }E$ with the coefficients $%
N^\beta_k$ induces a complex nonlinear connection on $\mathcal{T}^{\prime }E$
with the coefficients $N^\beta_\alpha$ given by \eqref{7.6}. Moreover, the
relations between the adapted frames on $T^{\prime }E$ and $\mathcal{T}%
^{\prime }E$ are 
\begin{equation}  \label{7.7}
\rho_{\mathcal{T}}(\delta_\alpha) = \rho^k_\alpha\dfrac{\delta}{\delta z^k}.
\end{equation}
\end{proposition}

\begin{proof}
We first have to prove that $N^\beta_\alpha$ from \eqref{7.6} are the coefficients of a complex nonlinear connection on $\mathcal{T}'E$. Let $\rho^*_{\mathcal{T}}$ be the dual map of $\rho_{\mathcal{T}}$. Then $\rho^*_{\mathcal{T}}(dz^k) = \rho^k_\alpha\mathcal{Z}^\alpha$ and $\rho^*_{\mathcal{T}}(du^\alpha) = \mathcal{V}^\alpha$, such that $\rho^*_{\mathcal{T}}(\delta u^\alpha) = \rho^*_{\mathcal{T}}(du^\alpha + N^\alpha_k dz^k) = \mathcal{V}^\alpha + N^\alpha_k\rho^k_\beta\mathcal{Z}^\beta$. On the other hand, on $(\mathcal{T}'E)^*$, the dual adapted frame is $\delta\mathcal{V}^\alpha = \mathcal{V}^\alpha + N^\alpha_\beta\mathcal{Z}^\beta$ and $\rho^*_{\mathcal{T}}|_{(V\mathcal{T}'E)^*}:(V\mathcal{T}'E)^*\rightarrow (V\mathcal{T}'E)^*$ is an isomorphism, that is, $\rho^*_{\mathcal{T}}(\delta u^\alpha) = \delta\mathcal{V}^\alpha$, which yields $N^\alpha_k\rho^k_\beta = N^\alpha_\beta$. Also,
\begin{align*}
\rho_{\mathcal{T}}(\delta_\alpha) &= \rho_{\mathcal{T}}(\mathcal{Z}_\alpha - N^\beta_\alpha\mathcal{V}_\beta) = \rho^k_\alpha\dfrac{\partial}{\partial z^k} - N^\beta_\alpha\dfrac{\partial}{\partial u^\beta}\\
&= \rho^k_\alpha\bigg(\dfrac{\partial}{\partial z^k} - N^\beta_k\dfrac{\partial}{\partial u^\beta}\bigg) = \rho^k_\alpha\dfrac{\delta}{\delta z^k}.
\end{align*}
\end{proof}

This result shows that the adapted frame \eqref{7.3} can very well be
interpreted as defining a new complex nonlinear connection on $E$. An
interesting property of this connection is that it can be derived from a
spray.

\begin{theorem}
If $G^\alpha$ are the coefficients of a complex spray on $T^{\prime }E$, as
defined in \eqref{3.7}, then the functions
\begin{equation}  \label{7.8}
N^\beta_\alpha = \dfrac{\partial G^\alpha}{\partial u^\alpha} +
P^\beta_\alpha
\end{equation}
define a complex nonlinear connection on $E$, where
\begin{equation}  \label{10}
P^\beta_\alpha = \dfrac{1}{4}W^\beta_\gamma\bigg(\rho^k_\alpha\dfrac{%
\partial M^\gamma_\delta}{\partial z^k}u^\delta - \dfrac{\partial
M^\gamma_\alpha}{\partial z^k}\rho^k_\delta u^\delta\bigg).
\end{equation}
\end{theorem}

\begin{proof}
Deriving in \eqref{3.7} with respect to $u^\gamma$ and taking into account the \eqref{sch.ro} rules of change yields
\begin{equation*}
M^\delta_\gamma\dfrac{\partial\widetilde{G}^\alpha}{\partial\widetilde{u}^\delta} = M^\alpha_\beta\dfrac{\partial G^\beta}{\partial u^\gamma} - \dfrac{1}{2}\dfrac{\partial M^\alpha_\gamma}{\partial z^k}\rho^k_\delta u^\delta - \dfrac{1}{2}\rho^k_\gamma\dfrac{\partial M^\alpha_\delta}{\partial z^k}u^\delta.
\end{equation*}
Replacing $\dfrac{\partial G^\beta}{\partial u^\gamma} = N^\beta_\gamma - P^\beta_\gamma$ gives
\begin{equation*}
M^\delta_\gamma\widetilde{N}^\alpha_\delta - M^\delta_\gamma\widetilde{P}^\alpha_\delta = M^\alpha_\beta N^\beta_\gamma - M^\alpha_\beta P^\beta_\gamma - \rho^k_\delta\dfrac{\partial M^\alpha_\delta}{\partial z^k}u^\delta + \dfrac{1}{2}\rho^k_\delta\dfrac{\partial M^\alpha_\delta}{\partial z^k}u^\delta - \dfrac{1}{2}\dfrac{\partial M^\alpha_\gamma}{\partial z^k}\rho^k_\delta u^\delta.
\end{equation*}
Comparing this with \eqref{7.5} means that we have to show that
\begin{equation}
\label{11}
M^\alpha_\beta P^\beta_\gamma - M^\delta_\gamma\widetilde{P}^\alpha_\delta = \dfrac{1}{2}\rho^k_\delta\dfrac{\partial M^\alpha_\delta}{\partial z^k}u^\delta - \dfrac{1}{2}\dfrac{\partial M^\alpha_\gamma}{\partial z^k}\rho^k_\delta u^\delta.
\end{equation}
First, from \eqref{10}, it follows that
\begin{equation}
\label{aux1}
M^\alpha_\beta P^\beta_\gamma = \dfrac{1}{4}\bigg( \rho^k_\delta\dfrac{\partial M^\alpha_\delta}{\partial z^k}u^\delta - \dfrac{\partial M^\alpha_\gamma}{\partial z^k}\rho^k_\delta u^\delta \bigg).
\end{equation}
Then, again using \eqref{10},
\begin{equation*}
\widetilde{P}^\alpha_\delta = \dfrac{1}{4}M^\alpha_\beta\bigg( \widetilde{\rho}^h_\delta\dfrac{\partial W^\beta_\varepsilon}{\partial\widetilde{z}^h}\widetilde{u}^\varepsilon - \dfrac{\partial W^\beta_\delta}{\partial\widetilde{z}^h}\widetilde{\rho}^h_\varepsilon\widetilde{u}^\varepsilon \bigg),
\end{equation*}
which, multiplied by $-M^\delta_\gamma$ and using \eqref{sch.ro} and $M^\alpha_\beta\dfrac{\partial W^\beta_\varepsilon}{\partial\widetilde{z}^h} = -\dfrac{\partial M^\alpha_\beta}{\partial\widetilde{z}^h}W^\beta_\varepsilon$, gives after some basic computations
\begin{equation}
\label{aux2}
-M^\delta_\gamma\widetilde{P}^\alpha_\delta = \dfrac{1}{4}\bigg( \rho^k_\delta\dfrac{\partial M^\alpha_\delta}{\partial z^k}u^\delta - \dfrac{\partial M^\alpha_\gamma}{\partial z^k}\rho^k_\delta u^\delta \bigg).
\end{equation}
Adding \eqref{aux1} and \eqref{aux2} yields \eqref{11}.
\end{proof}

\section{Induced Lagrange structures on holomorphic Lie algebroids}

\label{S7}

It is interesting to study the conditions under which two nonlinear
connections on $E$ and $T^{\prime }M$, respectively, are linked.

The geometry of the bundle $T^{\prime }M$ endowed with a complex Lagrangian $%
L(z,\eta )$, where 
\begin{equation}
g_{i\bar{j}}(z,\eta )=\frac{\partial ^{2}L}{\partial \eta ^{i}\partial \bar{%
\eta}^{j}}  \label{G}
\end{equation}
defines a nondegenerate metric, is well-known. The pair $(M,L)$ is called 
\emph{complex Lagrange } space (\cite{Mu}). In order to distinguish it from $%
E$, we will further denote it $(T^{\prime }M,L)$.

A remarcable complex nonlinear connection on $T^{\prime }M$ is 
\begin{equation}
N_{k}^{i}=g^{\bar{j}i}\frac{\partial ^{2}L}{\partial z^{k}\partial \bar{\eta}%
^{j}},  \label{C-L}
\end{equation}
called the Chern-Lagrange nonlinear connection. If $L$ if homogeneous in $%
\eta$, i.e. $L(z,\lambda \eta )=\lambda \bar{\lambda}L(z,\eta )$ for all $%
\lambda \in \mathbb{C},$ then $(M,L)$ is called complex Finsler space and in
particular (\ref{C-L}) defines a complex nonlinear connection on $T^{\prime
}M$, called Chern-Finsler nonlinear connection.

The problem of introducing a Lagrange (in particular, Finsler), on $E$ by a
Lagrange (Finsler) complex space $(T^{\prime }M,L)$ is relatively simple if
we impose some additional conditions on the fibers of $E$ and on the anchor
map $\rho $. We will now analyze the three possible cases, depending on the
relation between the dimensions of $M$, $E$ and the rank of $\rho$.

\medskip \textbf{I)} The case when $m=n=\opn{rank}\rho.$ \medskip

Recall that on the manifold $E$ we have local coordinates in a map $%
(z^{k},u^{\alpha })$, while on $T^{\prime }M$ we have $(z^{k},\eta ^{k}),$
where $\eta ^{k}=u^{\alpha }\rho _{\alpha }^{k}(z)$, as stated in %
\eqref{II.1}, and $\alpha,\beta,...,i,j,...$ all range over $\overline{1,n}$.

Since $n=\opn{rank}\rho$, it follows that $\rho $ is a diffeomorphism with $%
\rho^{-1}=[\rho _{k}^{\alpha }]$, such that $\rho _{\alpha }^{k}\rho
_{k}^{\beta}=\delta _{\alpha }^{\beta }$ and $u^{a}=\rho _{k}^{\alpha }\eta
^{k}$.

If $(T^{\prime }M,L)$ is a complex Lagrange (Finsler) complex space and $%
L(z,\eta )$ is the Lagrange function, then, according to \eqref{II.1}, it
induces on $\rho (E)$ another Lagrange function, $L^{\ast }(z,u)=L(z,\eta
(u))$, where the metric tensor is
\begin{equation*}
g_{\alpha \bar{\beta}}(z,u):=\frac{\partial ^{2}L^{\ast }}{\partial
u^{\alpha }\partial \bar{u}^{\beta }}=\rho _{\alpha }^{i}\rho _{\bar{\beta}%
}^{\bar{j}}g_{i\bar{j}}.
\end{equation*}
Since $\rho $ is a diffeomorphism, $\opn{rank}g_{\alpha \bar{\beta}}=n=m$
and $L^{\ast }$ can be interpreted as a function on $E$.

The pair $(E,L^{\ast })$ is called \emph{Lagrange structure} on the Lie
algebroid $E$. Notice that if $L$ is homogeneous, then $L^{\ast }(z,\lambda
u)=\lambda\bar{\lambda}L^{\ast }(z,u)$, in which case a \emph{Finsler
structure} is induced on $E$.

Let $N_{k}^{\alpha }(z,\eta )$ be the coefficients of a nonlinear connection
on $E$ and $\dfrac{\delta }{\delta z^{k}}=\dfrac{\partial }{\partial z^{k}}%
-N_{k}^{\alpha }\dfrac{\partial }{\partial u^{\alpha }}$ the corresponding
adapted frame of fields. It is mapped by $\rho _{\ast }$, using %
\eqref{act.apl.tg.}, to $\dfrac{\delta ^{\ast }}{\delta z^{k}}=\dfrac{%
\partial }{\partial z^{k}}-\overset{\ast }{N_{k}^{h}}\dfrac{\partial }{%
\partial \eta ^{h}}$ and we obtain 
\begin{equation}
\overset{\ast }{N_{k}^{h}}(z,\eta )=\rho _{\alpha }^{h}N_{k}^{\alpha }(z,u)-%
\dfrac{\partial \rho _{\alpha }^{h}}{\partial z^{k}}u^{\alpha }
\label{c.i.1}
\end{equation}
where $\eta ^{k}=u^{\alpha }\rho _{\alpha }^{k}(z)$. From $\dfrac{\delta }{%
\delta z^{h}}=\dfrac{\partial \widetilde{z}^{k}}{\partial z^{h}}\dfrac{%
\delta }{\delta \widetilde{z}^{k}}$, it follows that for a change of local
maps on $E$ we have $\dfrac{\delta ^{\ast }}{\delta z^{h}}=\dfrac{\partial 
\widetilde{z}^{k}}{\partial z^{h}}\dfrac{\delta ^{\ast }}{\delta \widetilde{z%
}^{k}}$, hence $\overset{\ast }{N_{k}^{h}}$ defines a nonlinear connection,
induced on $T^{\prime }M$ by $N_{k}^{\alpha }$.

Conversely, let $N_{k}^{h}$ be a complex nonlinear connection on $T^{\prime
}M$. We search for a nonlinear connection $\overset{\ast }{N_{k}^{\alpha }}$
induced on $E$ by $N_{k}^{h}$. Obviously, we will need that $%
N_{k}^{h}(z,\eta )=\rho _{\alpha }^{h}\overset{\ast }{N_{k}^{\alpha }}(z,u)-%
\dfrac{\partial \rho _{\alpha }^{h}}{ \partial z^{k}}u^{\alpha }$, which,
contracted with $\rho _{h}^{\beta }$ gives $\overset{\ast }{N_{k}^{\beta }}%
(z,u)=\rho _{h}^{\beta }N_{k}^{h}(z,\eta)+\rho _{h}^{\beta }\dfrac{\partial
\rho _{\alpha }^{h}}{\partial z^{k}}u^{\alpha }$. By deriving the condition $%
\rho _{\alpha }^{k}\rho _{k}^{\beta }=\delta _{\alpha }^{\beta }$ and
substituting in the above we get the following nonlinear connection induced
on $E$: 
\begin{equation}
\overset{\ast }{N_{k}^{\alpha }}(z,u)=\rho _{h}^{\alpha }N_{k}^{h}(z,\eta )-%
\dfrac{\partial \rho _{h}^{\alpha }}{\partial z^{k}}\eta ^{k}.  \label{c.i.2}
\end{equation}

To conclude, in the case when $m=n=\opn{rank}\rho$, the diffeomorphism $%
\rho _{\ast}$ maps the decomposition $T^{\prime }E=VE\oplus HE$ in $%
T^{\prime }T^{\prime}M=VT^{\prime }M\oplus HT^{\prime }M$ preserving the
distributions, and $\rho_{\ast }^{-1}$ has the converse role.

\begin{proposition}
If $(T^{\prime }M,L)$ is a complex Lagrange (Finsler) space and \eqref{C-L}
is an associated complex nonlinear connection, then 
\begin{equation}
\overset{\ast }{N_{k}^{\alpha }}=g^{\bar{\beta}\alpha }\frac{\partial
^{2}L^{\ast }}{\partial z^{k}\partial \bar{u}^{\beta }}  \label{c.i.C-L}
\end{equation}
is the nonlinear connection induced on $E$, called the Chern-Lagrange
connection of the Lie algebroid.
\end{proposition}

\medskip \textbf{II)} The case when $\opn{rank}\rho=m<n$. \medskip

Note that, in this case, the morphism $\rho $ maps $E$ in $\rho(E)$ which is
an immersed submanifold of $T^{\prime }M$.

As in the first case, we will introduce on $E$ Lagrange structures induced
by a Lagrange structure $(T^{\prime }M,L)$.

Let $g_{i\bar{j}}(z,\eta )=\dfrac{\partial^{2}L}{\partial \eta ^{i}\partial 
\bar{\eta}^{j}}$ the metric tensor defined by the regular Lagrangian $%
L:T^{\prime }M\rightarrow R$. As in the first case, we consider the
Lagrangian induced on $\rho (E)$ given by $L^{\ast }(z,u)=L(z,\eta (u))$,
with the metric tensor 
\begin{equation*}
g_{\alpha \bar{\beta}}(z,u):=\frac{\partial ^{2}L^{\ast }}{\partial
u^{\alpha }\partial \bar{u}^{\beta }}=\rho _{\alpha }^{i}\rho _{\bar{\beta}%
}^{\bar{j}}g_{i\bar{j}}.
\end{equation*}

Since $\opn{rank}\rho=m$ and $\opn{rank}[g_{i\bar{j}}]=n>m$, it follows
that $\opn{rank}[g_{\alpha \bar{\beta}}]=m$.

Let $X_{\alpha }=\rho _{\alpha }^{k}\dfrac{\partial }{\partial z^{k}}\in
\chi(M)$, $\alpha=\overline{1,m}$, be fields on the basis manifold $M$. From 
$\opn{rank}\rho=m $ we read that $\{X_{\alpha }\}$ are linear independent
and can be lifted to $\rho(E)$ by $\big\{X_{\alpha }^{\ast }:=\dfrac{%
\partial }{\partial u^{\alpha }}=\rho _{\alpha }^{k}\dfrac{\partial }{%
\partial \eta ^{k}}\big\}_{\alpha=\overline{1,m}}$, which defines an $m$%
-dimensional subdistribution $V\rho(E)$ of the $n$-dimensional distribution $%
VT^{\prime }M$.

Let us fix a complex nonlinear connection $N_{k}^{h}(z,\eta )$ on $T^{\prime
}M$, in particular the Chern-Lagrange connection, and let us consider its
adapted basis and cobasis. We search for a nonlinear connection $\overset{%
\ast }{N_{k}^{\alpha }}(z,\eta (u))$ induced by $N_{k}^{h}(z,\eta )$ as the
image through $\rho $ of a nonlinear connection $N_{k}^{\alpha }(z,u)$ on $E$.

Denote by $G=g_{i\bar{j}}\delta \eta^{i}\otimes \delta \bar{\eta}^{j}$ the
metric structure on $VT^{\prime }M$. We complete $\{X_{\alpha
}^{\ast}\}_{\alpha =\overline{1,m}}$ with $\{Y_{a}\}_{a=\overline{1,n-m}}$,
vector fields normal to $V\rho (E)$ with respect to $G$. Moreover, we assume
these vectors to be orthonormal. By writing $Y_{a}=Y_{a}^{k}\dfrac{\partial 
}{\partial \eta ^{k}}$, the orthogonality conditions read 
\begin{equation}
g_{i\bar{j}}Y_{a}^{i}\rho _{\bar{\alpha}}^{\bar{j}}=g_{i\bar{j}}\rho
_{\alpha }^{i}Y_{\bar{a}}^{\bar{j}}=0  \label{ort}
\end{equation}%
and the normality ones give $g_{i\bar{j}}Y_{a}^{i}Y_{\bar{b}}^{\bar{j}%
}=\delta _{ab}$.

We have thus obtained $\mathcal{R}^{\ast }=\{X_{\alpha }^{\ast },Y_{a}\}$, a
frame on $VT^{\prime }M$ with the matrix $R=[\rho _{\alpha }^{i}\;;\
Y_{a}^{i}]$ of change from the natural frame $\bigg\{\dfrac{\partial }{%
\partial \eta ^{k}}\bigg\}$. Let $R^{-1}=[\rho _{i}^{\alpha }\;;\
Y_{i}^{a}]^{t}$ its inverse matrix, such that 
\begin{equation}
\rho _{i}^{\alpha }\rho _{\beta }^{i}=\delta _{\beta }^{\alpha }\ ;\quad
\rho_{i}^{\alpha }Y_{a}^{i}=0\ ;\quad Y_{i}^{a}Y_{b}^{i}=0\ ;\quad \rho
_{\alpha}^{j}\rho _{i}^{\alpha }+Y_{a}^{j}Y_{i}^{a}=\delta _{i}^{j}.
\label{inv}
\end{equation}

The technique used in the following is known from holomorphic subspaces, 
\cite{Mu}. From \eqref{inv} we get 
\begin{equation}
g^{\bar{j}k}Y_{k}^{a}\rho _{\bar{j}}^{\bar{\alpha}}=0\ ;\quad g^{\bar{\beta}%
\alpha }=\rho _{i}^{\alpha }\rho _{\bar{j}}^{\bar{\beta}}g^{\bar{j}i}.
\label{inv.2}
\end{equation}

Now, let $N_{k}^{h}(z,\eta )$ be a complex nonlinear connection on $%
T^{\prime }M$ and $N_{k}^{\alpha }(z,u)$ a connection on $E$, with the
adapted cobases $\{dz^{k},\delta u^{\alpha }=du^{\alpha }+N_{h}^{\alpha
}dz^{h}\}$ and $\{dz^{k},\delta \eta ^{k}=d\eta ^{k}+N_{h}^{k}dz^{h}\}$,
respectively. The identities in \eqref{dual} suggest considering the cobasis 
$\delta \eta ^{k}$ as a frame of forms $\delta ^{\ast }\eta ^{k}=d^{\ast
}\eta ^{k}+\overset{\ast }{N_{h}^{k}}d^{\ast }z^{h}$ induced on $\rho (E)$,
where $\overset{\ast }{N_{h}^{k}}=N_{h}^{k}(z,\eta )$ with $\eta ^{k}=\rho
_{\alpha }^{k}u^{\alpha }$. Among these, only $m$ are linear independent. In
the following, for the simplicity of writing, we will identify $\overset{\ast 
}{N_{h}^{k}}=N_{h}^{k}$.

\begin{definition}
$N_{k}^{\alpha }$ is called an induced connection on $E$ by $N_{k}^{h}$ from 
$T^{\prime }M$ if 
\begin{equation}
\delta u^{\alpha }=\rho _{k}^{\alpha }\delta ^{\ast }\eta ^{k}.
\label{indusa}
\end{equation}
\end{definition}

In matrix form, the right-hand side yields $\opn{rank}[\delta u^{\alpha
}]=m $, thus $\{\delta u^{\alpha }\}$ can define a cobasis on $E$.
Substituting from \eqref{dual}, we obtain 
\begin{equation*}
du^{\alpha }+N_{h}^{\alpha }dz^{h} = \rho _{k}^{\alpha }\bigg\{u^{\beta }%
\dfrac{\partial \rho _{\beta }^{k}}{\partial z^{h}}dz^{h}+\rho
_{\beta}^{k}du^{\beta }+N_{h}^{k}dz^{h}\bigg\},
\end{equation*}
such that $\rho _{k}^{\alpha}\rho _{\beta }^{k}=\delta _{\beta }^{\alpha }$
yields $du^{\alpha}+N_{h}^{\alpha }dz^{h}=du^{\alpha }+\rho _{k}^{\alpha }%
\bigg\{N_{h}^{k}+u^{\beta}\dfrac{\partial \rho _{\beta }^{k}}{\partial z^{h}}%
\bigg\}dz^{h}$. This means that the connections are linked by 
\begin{equation}
N_{h}^{\alpha }(z,u)=\rho _{k}^{\alpha }\bigg\{N_{h}^{k}(z,\eta )+u^{\beta }%
\dfrac{\partial \rho _{\beta }^{k}}{\partial z^{h}}\bigg\},\quad \mbox{where}\
\eta ^{k}=\rho _{\alpha }^{k}u^{\alpha }.  \label{c.n. ind}
\end{equation}

Thus, given a nonlinear connection $N_{h}^{k}(z,\eta )$ on $T^{\prime }M$,
we have obtained a nonlinear connection $N_{h}^{\alpha }(z,u)$ on $E$. The
relations between the two connections are given in the following

\begin{proposition}
If $N_{h}^{\alpha }(z,u)$ is the nonlinear connection induced on $E$ by the
nonlinear connection $N_{h}^{k}(z,\eta )$ from $T^{\prime }M$, then
\begin{description}
\item[i)] $dz^{k} = d^{\ast }z^{k}\ ;\quad \delta^{\ast }\eta^{k} =\rho
_{\alpha }^{k}\delta u^{\alpha}+Y_{a}^{k}Y_{j}^{a}H_{h}^{j}dz^{h}$, where $%
H_{h}^{j}=N_{h}^{j}+u^{\beta }\dfrac{\partial \rho _{\beta }^{j}}{\partial
z^{h}}$.

\item[ii)] $\dfrac{\delta ^{\ast }}{\delta z^{k}} = \dfrac{\delta }{\delta
z^{k}}+Y_{k}^{a}Y_{a}^{h}H_{h}^{j}\frac{\partial }{\partial \eta ^{j}}\
;\quad \dfrac{\partial ^{\ast }}{\partial u^{\alpha }}=\rho _{\alpha }^{k}%
\dfrac{\partial }{\partial \eta ^{k}}$.
\end{description}
\end{proposition}

\begin{proof}
The first point is already proven, thus we only have to prove \textbf{ii)}:
\begin{align*}
\dfrac{\delta ^{\ast }}{\delta z^{k}} &= \dfrac{\partial ^{\ast }}{\partial z^{k}} - N_{k}^{\alpha }\frac{\partial ^{\ast }}{\partial u^{\alpha }}\\
&= \dfrac{\partial }{\partial z^{k}}+u^{\alpha }\dfrac{\partial \rho _{\alpha }^{h}}{\partial z^{k}}\dfrac{\partial }{\partial \eta ^{h}}-N_{k}^{\alpha }\rho_{\alpha }^{h}\dfrac{\partial }{\partial \eta ^{h}}\\
&= \dfrac{\delta }{\delta z^{k}}+N_{k}^{h}\dfrac{\partial }{\partial \eta ^{h}}+u^{\alpha }\dfrac{\partial \rho _{\alpha }^{h}}{\partial  z^{k}}\dfrac{\partial }{\partial \eta ^{h}}-\rho _{\alpha }^{h}\rho _{k}^{\alpha}\bigg\{N_{h}^{j} + u^{\beta }\dfrac{\partial \rho _{\beta }^{j}}{\partial z^{h}}\bigg\}\dfrac{\partial }{\partial \eta ^{j}}.
\end{align*}
Using $\rho _{\alpha }^{h}\rho _{k}^{\alpha }=\delta_{k}^{h}-Y_{k}^{a}Y_{a}^{h}$ yields
$\dfrac{\delta ^{\ast }}{\delta z^{k}}=\dfrac{\delta }{\delta z^{k}}
+Y_{k}^{a}Y_{a}^{h}H_{h}^{j}\dfrac{\partial }{\partial \eta ^{j}}$.
\end{proof}

Let us consider on $T^{\prime }M$ the Chern-Lagrange connection, $%
N_{k}^{i}=g^{\bar{j}i}\dfrac{\partial ^{2}L}{\partial z^{k}\partial \bar{\eta%
}^{j}}$, while on $\rho (E)$ we consider a nonlinear connection of
Chern-Lagrange type, 
\begin{equation}
N_{k}^{\alpha }=g^{\bar{\beta}\alpha }\frac{\partial ^{2}L^{\ast }}{\partial
^{\ast }z^{k}\partial ^{\ast }\bar{u}^{\beta }}.  \label{C-L ind}
\end{equation}
Then 
\begin{equation*}
N_{k}^{\alpha }=\rho _{i}^{\alpha }\rho _{\bar{j}}^{\bar{\beta}}g^{\bar{j}i}%
\dfrac{\partial ^{\ast }}{\partial z^{k}}\left( \rho _{\bar{\beta}}^{\bar{h}}%
\frac{\partial L}{\partial \bar{\eta}^{h}}\right) =\rho _{i}^{\alpha }\rho _{%
\bar{j}}^{\bar{\beta}}g^{\bar{j}i}\left( \dfrac{\partial }{\partial z^{k}}%
+u^{\gamma }\dfrac{\partial \rho _{\gamma }^{l}}{\partial z^{k}}\dfrac{%
\partial }{\partial \eta ^{l}}\right) \left( \rho _{\bar{\beta}}^{\bar{h}}%
\dfrac{\partial L}{\partial \bar{\eta}^{h}}\right)
\end{equation*}
and since $\rho $ is holomorphic, hence $\dfrac{\partial }{\partial z^{k}}%
\big(\rho _{\bar{\beta}}^{\bar{h}}\big)=0$, we get 
\begin{equation*}
N_{k}^{\alpha }=\rho _{i}^{\alpha }\rho _{\bar{j}}^{\bar{\beta}}g^{ \bar{j}%
i}\left( \rho _{\bar{\beta}}^{\bar{h}}\dfrac{\partial ^{2}L}{\partial
z^{k}\partial \bar{\eta}^{h}}+u^{\gamma }\dfrac{\partial \rho _{\gamma }^{l}%
}{\partial z^{k}}\rho _{\bar{\beta}}^{\bar{h}}g_{l\bar{h}}\right)
=\rho_{i}^{\alpha }\rho _{\bar{j}}^{\bar{\beta}}\rho _{\bar{\beta}}^{\bar{h}%
}g^{\bar{j}i}\left( \dfrac{\partial ^{2}L}{\partial z^{k}\partial \bar{\eta}%
^{h}}+u^{\gamma }\dfrac{\partial \rho _{\gamma }^{l}}{\partial z^{k}}g_{l%
\bar{h}}\right).
\end{equation*}
With $\rho _{\bar{j}}^{\bar{\beta}}\rho _{\bar{\beta}}^{\bar{h}}=\delta_{%
\bar{j}}^{\bar{h}}-Y_{\bar{j}}^{\bar{a}}Y_{\bar{a}}^{\bar{h}}$, it follows
that 
\begin{equation*}
N_{k}^{\alpha }=\rho _{i}^{\alpha }g^{\bar{h}i}\frac{\partial ^{2}L}{%
\partial z^{k}\partial \bar{\eta}^{h}}+\rho _{i}^{\alpha }u^{\gamma }\dfrac{%
\partial \rho _{\gamma }^{i}}{\partial z^{k}}-\rho _{i}^{\alpha }Y_{\bar{j}%
}^{\bar{a}}Y_{\bar{a}}^{\bar{h}}g^{\bar{j}i}\left( \dfrac{\partial ^{2}L}{%
\partial z^{k}\partial \bar{\eta}^{h}}+u^{\gamma }\dfrac{\partial
\rho_{\gamma }^{l}}{\partial z^{k}}g_{l\bar{h}}\right).
\end{equation*}
Now, using \eqref{inv.2}, we obtain $N_{k}^{\alpha }=\rho _{i}^{\alpha
}N_{k}^{\alpha }=\rho _{i}^{\alpha }\left( N_{k}^{i}+u^{\gamma }\dfrac{%
\partial \rho_{\gamma }^{i}}{\partial z^{k}}\right)$, which, compared to %
\eqref{c.n. ind}, gives the following

\begin{proposition}
The nonlinear connection induced on $E$ by the Chern-Lagrange given by %
\eqref{C-L} from $T^{\prime }M$, coincides with \eqref{C-L ind}.
\end{proposition}

\medskip \textbf{III)} The case when $\opn{rank}\rho=n<m$. \medskip

In this case, $\rho $ is a submersion and $\rho(E)$ can be identified with $%
T^{\prime }M$. We will introduce a Lagrange structure on $T^{\prime }M$
induced by a Lagrange structure on $E$.

Let $L(z,u)$ be a regular Lagrangian on $E$ with the metric tensor $%
g_{\alpha \bar{\beta}}(z,u)=\dfrac{\partial ^{2}L}{\partial
u^{\alpha}\partial \bar{u}^{\beta }}$, with $\det [g_{\alpha \bar{\beta}%
}]\neq 0$, such that $\opn{rank}g_{\alpha \bar{\beta}}=m$. Let $g^{\bar{%
\beta}\alpha }$ be the inverse of the metric tensor, that is, $g^{\bar{\beta}%
\alpha }g_{\gamma \bar{\beta}}=\delta _{\gamma }^{\alpha }$.

On $T^{\prime }M$ we have the induced frame \eqref{act.apl.tg.} and its
coframe \eqref{dual}. The Lagrangian $L(z,u)$ will be mapped by $\rho $ in 
$L^{\ast }(z,\eta ^{k}=\rho _{\alpha }^{k}u^{\alpha })$ and using %
\eqref{act.apl.tg.} we compute 
\begin{equation*}
g_{\alpha \bar{\beta}}^{\ast } = \dfrac{\partial ^{\ast 2}L^{\ast }}{%
\partial u^{\alpha }\partial \bar{u}^{\beta }}=\rho _{\alpha }^{i}\rho _{%
\bar{\beta}}^{\bar{j}}g_{i\bar{j}},\quad\mbox{where}\quad g_{i\bar{j}%
}(z,\eta ) = \dfrac{\partial ^{2}L^{\ast }}{\partial \eta ^{i}\partial \bar{%
\eta}^{j}}.
\end{equation*}
We obtain $\opn{rank}[g_{i\bar{j}}]=n$, since $\opn{rank}\rho=n$.

On $E$, we define the following $n$ vertical, linear independent forms
\begin{equation*}
dv^{k}=\rho _{\alpha }^{k}du^{\alpha }.
\end{equation*}

They form an $n$-dimensional distribution, denoted by $\rho
^{\-1}T^{\prime\ast }M$, of the $m$-dimensional distribution $V^{\prime \ast
}E$, and, according to \eqref{act.apl.tg.}, are linked with their image
through $\rho _{\ast }$ on $T^{\prime \ast }M$ by 
\begin{equation}
dz^{k}=dz^{\ast k}\ ;\quad dv^{k}=d^{\ast }\eta ^{k}-u^{\alpha }\dfrac{%
\partial \rho _{\alpha }^{k}}{\partial z^{h}}dz^{h},  \label{cob. ind.}
\end{equation}
where $\eta ^{k}=\rho _{\alpha }^{k}u^{\alpha }$.

We now complete $\{dv^{k}\}_{k=\overline{1,n}}$ to a cobasis of $V^{\prime
\ast }E$ with $\{dy^{a}=Y_{\alpha }^{a}du^{\alpha }\}_{a=\overline{1,m-n}}$.
We obtain the matrix of this frame, $R=\left[ \rho _{\alpha }^{k}\;;\
Y_{\alpha }^{a}\right]^{t}$. If $G_{E}=g_{\alpha \bar{\beta}}du^{\alpha
}\otimes d\bar{u}^{\beta }$ is a metric in the vertical bundle, using the
isomorphism of the tangent and cotangent spaces, let $G_{E}^{-1}=g^{\bar{%
\beta}\alpha }\dfrac{\partial }{\partial u^{\alpha }}\otimes \dfrac{\partial 
}{\partial \bar{u}^{\alpha }}$ be its action on $V^{\prime \ast }E$.

Using the same idea as in the second case, we impose that $\{dy^{a}\}$ are
normal forms with respect to $G_{E}^{-1}$ on $\rho ^{\-1}T^{\prime \ast }M$
and also that they are orthonormal, that is, 
\begin{equation*}
g^{\bar{\beta}\alpha }Y_{\alpha }^{a}\rho _{\bar{\beta}}^{\bar{k}}=g^{\bar{%
\beta}\alpha }\rho _{\alpha }^{k}Y_{\bar{\alpha}}^{\bar{\beta}}=0\quad %
\mbox{and}\quad g^{\bar{\beta}\alpha }Y_{\alpha }^{a}Y_{\bar{\beta}}^{\bar{b}%
}=\delta ^{ab}.
\end{equation*}

We have thus obtained a different coframe from the natural one, $%
\{dv^{k},dy^{a}\}$ on $V^{\prime \ast }E$. Let $R=\left[ \rho _{k}^{\alpha
}\;;\ Y_{a}^{\alpha }\right]$ be the inverse matrix of $R$, 
\begin{equation*}
\rho _{k}^{\alpha }\rho _{\alpha }^{h}=\delta _{k}^{h}\ ;\quad
\rho_{k}^{\alpha }Y_{\alpha }^{a}=0\ ;\quad Y_{\alpha }^{k}Y_{a}^{\alpha
}=0\ ;\quad \rho _{\alpha }^{i}\rho _{i}^{\beta }+Y_{\alpha
}^{a}Y_{a}^{\beta }=\delta_{\alpha }^{\beta }.
\end{equation*}

The first identity yields that $\bigg\{\dfrac{\partial }{\partial z^{k}},%
\dfrac{\partial }{\partial v^{k}}:=\rho _{k}^{\alpha }\dfrac{\partial }{%
\partial u^{\alpha }}\bigg\}$ is the dual frame of $\{dz^{k},dv^{k}\}$ on $%
\rho^{\-1}T^{\prime \ast }M$, since $dz^{k}\bigg(\dfrac{\partial }{\partial
z^{h}}\bigg)=\delta _{h}^{k}$. Indeed, using \eqref{cob. ind.}, we have $%
dz^{k}\bigg(\dfrac{\partial }{\partial z^{h}}\bigg)=0$, since $d^{\ast }\eta
^{k}\bigg(\dfrac{\partial }{\partial z^{h}}\bigg)=u^{\alpha }\dfrac{\partial
\rho _{\alpha }^{k}}{\partial z^{h}}$. Also, $dv^{k}\bigg(\dfrac{\partial }{%
\partial v^{h}}\bigg)=\rho _{h}^{\alpha }\rho _{\alpha }^{k}=\delta _{h}^{k}$
and $dv^{k}\bigg(\dfrac{\partial }{\partial z^{k}}\bigg)=0$.

From \eqref{act.apl.tg.}, we have 
\begin{align*}
\rho _{\ast }\bigg(\dfrac{\partial }{\partial z^{k}}\bigg) &=: \dfrac{\partial ^{\ast }%
}{\partial z^{k}}=\dfrac{\partial }{\partial z^{k}}+u^{\alpha }\dfrac{%
\partial \rho _{\alpha }^{h}}{\partial z^{k}}\dfrac{\partial }{\partial
\eta^{h}} \\
\rho _{\ast }\bigg(\dfrac{\partial }{\partial v^{k}}\bigg) &=: \dfrac{\partial ^{\ast }%
}{\partial v^{k}}=\rho _{k}^{\alpha }\dfrac{\partial ^{\ast }}{\partial
u^{\alpha }}=\rho _{k}^{\alpha }\rho _{\alpha }^{h}\dfrac{\partial }{%
\partial \eta ^{h}}=\dfrac{\partial }{\partial \eta ^{k}}
\end{align*}
Thus, 
\begin{equation}
g_{i\bar{j}}(z,\eta )=\frac{\partial ^{2}L^{\ast }}{\partial \eta
^{i}\partial \bar{\eta}^{j}}=\rho _{i}^{\alpha }\rho _{\bar{j}}^{\bar{\beta}}%
\frac{\partial ^{\ast 2}L^{\ast }}{\partial u^{\alpha }\partial \bar{u}%
^{\beta }}=\frac{\partial ^{\ast 2}L^{\ast }}{\partial v^{i}\partial \bar{v}%
^{j}}:=\left( g_{i\bar{j}}(z,u)\right) ^{\ast }  \label{gij}
\end{equation}
where $\eta ^{k}=\rho _{\alpha }^{k}u^{\alpha }$.

Let $N_{k}^{\alpha }(z,u)$ be a nonlinear connection on $E$ and $%
N_{k}^{h}(z,\eta )$ a nonlinear connection on $T^{\prime }M$, with the
adapted cobases $\delta u^{\alpha }=du^{\alpha}+N_{h}^{\alpha }dz^{h}$ and $%
\delta \eta ^{k}=d\eta^{k}+N_{h}^{k}dz^{h}$, respectively. Let $(z^{k},\eta
^{k}=\rho _{\alpha }^{k}u^{\alpha })$ be the coordinates induced on $%
T^{\prime }M$ by $\rho $ and $\delta^{\ast }\eta ^{k}=d^{\ast }\eta
^{k}+N_{h}^{k}(z,\rho _{\alpha}^{k}u^{\alpha })dz^{h}$.

\begin{definition}
The nonlinear connection $N_{h}^{k}(z,\eta )$ is called induced on $%
\rho(E)\equiv T^{\prime }M$ if 
\begin{equation}
\delta ^{\ast }\eta ^{k}=\rho _{\alpha }^{k}\delta u^{\alpha }.
\label{III.c.n}
\end{equation}
\end{definition}

Using \eqref{cob. ind.}, we obtain $dv^{k}+u^{\alpha }\dfrac{\partial \rho
_{\alpha }^{k}}{\partial z^{h}}dz^{h}+N_{h}^{k}(z,\rho_{\alpha
}^{k}u^{\alpha })dz^{h}=\rho _{\alpha }^{k}\left( du^{\alpha}+N_{h}^{\alpha
}dz^{h}\right)$. With $dv^{k}=\rho _{\alpha }^{k}du^{\alpha}$, this yields

\begin{proposition}
The nonlinear connection $N^k_h(z,\eta)$ is induced on $\rho (E)$ by the
nonlinear connection $N_{k}^{\alpha }(z,u)$ from $E$ if and only if 
\begin{equation}
N_{h}^{k}(z,\eta )=\rho _{\alpha }^{k}N_{h}^{\alpha }(z,u)-u^{\alpha }\dfrac{%
\partial \rho _{\alpha }^{k}}{\partial z^{h}},  \label{III.cn.ind}
\end{equation}
where $\eta ^{k}=\rho_{\alpha }^{k}u^{\alpha }$.
\end{proposition}

We now compute $\dfrac{\delta ^{\ast }}{\delta z^{k}}=\dfrac{\partial ^{\ast
}}{\partial z^{k}}-N_{k}^{\alpha }\dfrac{\partial ^{\ast }}{\partial
u^{\alpha }} =\dfrac{\partial }{\partial z^{k}}+u^{\alpha }\dfrac{\partial
\rho _{\alpha}^{h}}{\partial z^{k}}\dfrac{\partial }{\partial \eta ^{h}}%
-N_{k}^{\alpha}\rho _{\alpha }^{h}\dfrac{\partial }{\partial \eta ^{h}}$ and
replacing $N_{k}^{\alpha }\rho _{\alpha }^{h}$ from \eqref{III.cn.ind} gives 
$\dfrac{\delta ^{\ast }}{\delta z^{k}}=\dfrac{\delta }{\delta z^{k}}$ on $%
T^{\prime}M$.

Also, using \eqref{dual}, $\delta \eta ^{k}=d\eta
^{k}+N_{h}^{k}dz^{h}=d^{\ast }\eta^{k}+\left( \rho _{\alpha
}^{k}N_{h}^{\alpha }(z,u)-u^{\alpha }\dfrac{\partial \rho _{\alpha }^{k}}{%
\partial z^{h}}\right) dz^{h} =\rho _{\alpha}^{k}\delta u^{\alpha }=\delta
^{\ast }\eta ^{k}$.

\begin{proposition}
The adapted basis and cobasis of the nonlinear connection $N_{k}^{h}(z,\eta)$
induced by $N_{k}^{\alpha }(z,u)$ define the adapted frame and coframe on $%
T^{\prime }M$ given by 
\begin{align*}
\dfrac{\delta }{\delta z^{k}} &= \dfrac{\delta ^{\ast }}{\delta z^{k}}\
;\quad \dfrac{\partial }{\partial \eta ^{k}}=\rho _{k}^{\alpha }\dfrac{%
\partial ^{\ast }}{\partial u^{\alpha }}=\dfrac{\partial ^{\ast }}{\partial
v^{k}} \\
dz^{k} &= d^{\ast }z^{k}\ ;\quad \delta \eta ^{k}=\rho _{\alpha }^{k}\delta
u^{\alpha }=\delta ^{\ast }\eta ^{k}.
\end{align*}
\end{proposition}

\begin{proposition}
The Chern-Lagrange nonlinear connection from $E$ induces the Chern-Lagrange
connection on $T^{\prime }M$.
\end{proposition}

\begin{proof}
It follows in a similar manner an in the second case.
\end{proof}


\begin{thebibliography}{99}
\bibitem{A} T. Aikou, \emph{Finsler geometry on complex vector bundles},
Riemann Finsler Geometry, MSRI Publications, 50, p. 85-107, 2004.

\bibitem{A1} M. Anastasiei, \emph{Geometry of Lagrangians and semisprays on
Lie algebroids}, BSG Proc., 13, p. 10-17, Geom. Balkan Press, Bucharest,
2006.

\bibitem{A2} M. Anastasiei, \emph{Semisprays On Lie Algebroids. Applications}%
, Tensor (N.S.) 69, p. 190-198, 2008.

\bibitem{B-K} J. Bland, M. Kalka, \emph{Variations of holomorphic curvature
for Kahler Finsler metrics} Cont. Math, 1996, 196, p. 121--132.

\bibitem{I-Po} C. Ida, P. Popescu, \emph{On Almost Complex Lie Algebroids},
Mediterr. J. Math. DOI 10.1007/s00009-015-0516-4, 2015.

\bibitem{I} A. Ionescu, \emph{On holomorphic Lie algebroids}, Bulletin of
Transilvania Univ., Vol 9(58), No. 1 - 2016, to appear.

\bibitem{L-S-X} C. Laurent-Gengoux, M. Sti\' enon, P. Xu, \emph{Holomorphic
Poisson manifolds and holomorphic Lie algebroids}, Int. Math. Res. Not IMRN, 
ID:rnn088, \textbf{46}, 2008, arXiv:0803.2031.

\bibitem{Ma1} E. Martinez, \emph{Lagrangian mechanics on Lie algebroids},
Acta Applicandae Mathematicae, 67, p. 295-320, 2001.

\bibitem{Ma2} E. Martinez, \emph{Geometric formulation of mechanics on Lie
algebroids}, Proc. of the VIIIth Workshop on Geometry and Physics (Medina
del Campo, 1999), vol. 2 of Publ. R. Soc. Mat. Esp., p. 209-222, 2001.

\bibitem{Mu} Gh. Munteanu, \emph{Complex spaces in Finsler, Lagrange and
Hamilton geometries}, Kluwer Acad. Publishers, Dordrecht, 2004.

\bibitem{P} E. Peyghan, \emph{Models of Finsler Geometry on Lie algebroids},
arXiv:1310.7393v1, 2013.

\bibitem{P1} L. Popescu, \emph{On the geometry of Lie algebroids and
applications to optimal control}, An. \c St. ale Univ. "Al. I. Cuza", Ia\c
si, Vol. LI, s. I, 2005.

\bibitem{P2} L. Popescu, \emph{Geometrical structures on Lie algebroids},
Publ. Math. Debrecen 72 (1-2), p. 95-109, 2008.

\bibitem{Po} E. Popovici, \emph{\ On the volume of the indicatrix of a
complex Finsler space}, Turkish J. of Mathematics, to appears, DOI:
10.3906/mat-1510-66.

\bibitem{W} A. Weinstein, \emph{Lagrangian mechanics and grupoids}, Fields
Inst. Comm., 7, p. 206-231, 1996.

\bibitem{W1} A. Weinstein, \emph{The integration problem for complex Lie
algebroids}, arXiv:math/0601752, 2006.
\end{thebibliography}
\end{document}